\documentclass[11pt]{article}
\usepackage{t1enc}
\usepackage[latin1]{inputenc}
\usepackage[english]{babel}
\usepackage{amsmath,amsthm}
\usepackage{amsfonts}
\usepackage{latexsym}
\usepackage[dvips]{graphicx}
\usepackage{graphicx}
\usepackage{color}

 \usepackage{mathrsfs}

\newcounter{countclaim}

\usepackage{amsmath}
\usepackage{amssymb}
\usepackage{amsbsy}

\usepackage{amsfonts}

\usepackage{amsfonts,srcltx,mathrsfs}

\title{The extremal unicyclic graphs of the revised edge Szeged index with given diameter }

\author{Shengjie  He$^{\rm a}$\footnote{Corresponding author.
Emails: he1046436120@126.com (Shengjie  He), lxygqzh@tjcu.edu.cn (Qiaozhi Geng), rxhao@bjtu.edu.cn (Rong-Xia Hao)}, Qiaozhi Geng$^{\rm a}$, Rong-Xia Hao$^{\rm b}$\\
{\small\em $^{\rm a}$School of Science, Tianjin University of Commerce, Tianjin, 300134, China}\\
{\small\em $^{\rm b}$Department of Mathematics, Beijing Jiaotong University, Beijing,
100044, China}\\
}

\date{} \textwidth 16cm \textheight 22cm \topmargin 0 cm \hoffset
-1.5 cm \voffset 0cm

\newtheorem{theorem}{Theorem}[section]
\newtheorem{lemma}[theorem]{Lemma}

\newtheorem{definition}[theorem]{Definition}
\newtheorem{corollary}[theorem]{Corollary}

\begin{document}
\baselineskip 0.50cm \maketitle

\begin{abstract}

Let $G$ be a connected graph. The revised edge Szeged index of $G$ is defined as $Sz^{\ast}_{e}(G)=\sum\limits_{e=uv\in E(G)}(m_{u}(e|G)+\frac{m_{0}(e|G)}{2})(m_{v}(e|G)+\frac{m_{0}(e|G)}{2})$, where $m_{u}(e|G)$ (resp., $m_{v}(e|G)$) is the number of edges whose distance to vertex $u$ (resp., $v$) is smaller than the distance to vertex $v$ (resp., $u$), and $m_{0}(e|G)$ is the number of edges equidistant from both ends of $e$, respectively.
In this paper, the graphs with minimum revised edge Szeged index among all the unicyclic graphs with given diameter are characterized.

\vspace{0.2cm}

\noindent{\bf Keywords}: Edge Szeged index; Revised edge Szeged index; Unicyclic graph; Diameter.
\vspace{0.2cm}

\end{abstract}

\section{Introduction}

Throughout this paper, all graphs we considered are simple, undirected and connected. Let $G=(V(G), E(G))$ be a connected graph, where $V(G)$ and $E(G)$ are the vertex set and edge set of $G$, respectively.
For $u \in V(G)$, denote $d_{G}(u)$ the degree of $u$ in $G$.
A vertex $u$ is called a pendent vertex of $G$ if $d_{G}(u)=1$. An edge $uv$ is called a pendent edge of $G$ if $d_{G}(u)=1$ or $d_{G}(v)=1$.
Let $d(u, v|G)$ denote the distance between $u$ and $v$ in $V(G)$. The diameter of a graph $G$ is the maximum distance between pairs of vertices of $G$.
If $a$ and $b$ are two natural numbers with $a \leq b$, $[a, b]$ denote the set
$\{  n \in \mathbb{N}|a \leq n \leq b \}$, where $\mathbb{N}$ represent the set of natural numbers.
We refer to \cite{BONDY} for undefined terminologies and
notation.

The topological indices can be used in theoretical chemistry for understanding the physicochemical properties of chemical compounds.
The atoms and bonds of molecules can be represented by the vertices and edges of graphs, respectively.
The chemist Harold Wiener introduced the first topological index, named Wiener index, for investigating
boiling points of alkanes \cite{WIENER}. The Wiener index of a connected graph $G$ is defined as
$$W(G)=\sum\limits_{\{ u, v\} \subseteq V(G) } d(u, v|G).$$
For an edge $e=uv$ of $G$, the vertex set $V(G)$ can be partitioned into three sets as follows:
\begin{eqnarray*}N_{u}(e|G)&=&\{ w\in V(G): d(u, w|G) < d(v, w|G) \},\\
N_{v}(e|G)&=&\{ w\in V(G): d(v, w|G) < d(u, w|G) \},\\
N_{0}(e|G)&=&\{ w\in V(G): d(u, w|G) = d(v, w|G) \}.\end{eqnarray*}
Denote the number of vertices of $N_{u}(e|G)$, $N_{v}(e|G)$ and $N_{0}(e|G)$ by $n_{u}(e|G)$, $n_{v}(e|G)$ and $n_{0}(e|G)$, respectively.
It is known that $W(G)=\sum\limits_{e=uv \in E(G)}n_{u}(e|G)n_{v}(e|G)$ for $G$ is a acyclic graph.
Hence, Gutman \cite{Gut.A} introduced a new topological index, named by Szeged index, which was an extension of the Wiener index and defined by
$$Sz(G)=\sum\limits_{e=uv \in E(G)}n_{u}(e|G)n_{v}(e|G).$$
The Szeged index does not consider the vertices with equal distances from the endpoints of an edge. A modified version of the Szeged index was
introduced by Randi\'c \cite{RANDIC} which was named the revised Szeged index. The revised Szeged index of a connected graph $G$ is defined as
$$Sz^{\ast}(G)=\sum\limits_{e=uv\in E(G)}(n_{u}(e|G)+\frac{n_{0}(e|G)}{2})(n_{v}(e|G)+\frac{n_{0}(e|G)}{2}).$$

If $e=uv$ and $f$ are two edges of $G$ and $w$ is a vertex of $G$, then the distance between $e$ and $w$ is defined as $d(e,w|G) = {\rm{min}} \{ d(u,w|G),d(v,w|G) \}$, and the distance between $e$ and $f$ is defined as $d(e,f|G) = {\rm{min}} \{ d(u,f|G),d(v,f|G) \}$.
For $e=uv \in E(G)$, let $M_u(e|G)$ be the set of edges whose distance to the vertex $u$ is smaller than the distance to the vertex $v$, $M_v(e|G)$ be the set of edges whose distance to the vertex $v$ is smaller than the distance to the vertex $u$, and $M_{0}(e|G)$ be the set of edges equidistant from both ends of $e$. Set $m_u(e|G)=|M_u(e|G)|$, $m_v(e|G)=|M_v(e|G)|$ and $m_0(e|G)=|M_0(e|G)|$.

The edge version of the Wiener index, Szeged index and revised Szeged index are also introduced, and named by edge Wiener index, edge Szeged index and revised edge Szeged index, respectively. The edge Wiener index \cite{EWiener} of a graph $G$ is defined as follows:
$$W_{e}(G)= \sum\limits_{ \{e, f \} \subseteq E(G)} d(e, f|G).$$
The edge Szeged index of a graph $G$ is introduced by Gutman and Ashrafi \cite{Gut.A.R}, and defined as
$$Sz_{e}(G)=\sum\limits_{uv \in E(G)}m_{u}(uv|G)m_{v}(uv|G).$$
The edge revised Szeged index \cite{HDZB} of a graph $G$ is defined as:
$$Sz^{\ast}_{e}(G)=\sum\limits_{e=uv\in E(G)}(m_{u}(e|G)+\frac{m_{0}(e|G)}{2})(m_{v}(e|G)+\frac{m_{0}(e|G)}{2}).$$

Gutman \cite{Gut.A.R} established some basic properties of the edge Szeged index.
Li and Liu \cite{LXLLMM} discussed the bicyclic graphs with maximal revised Szeged index.
In \cite{HDZB}, Zhou et al. determined the $n$-vertex unicyclic graphs with the largest and the smallest revised edge Szeged indices.
Nadjafi-Arani et al. \cite{EDGE.Relation} proved that for every connected graph $G$, $Sz_{e}(G) \geq W_{e}(G)$ with equality if and only if $G$ is a tree.
The minimal and the second minimal revised edge Szeged indices of cacti with order $n$ and $k$ cycles were given by Liu and Wang \cite{LMM}, and all the cactus that achieve the minimal and second minimal revised edge Szeged index were identified. For other results on the Wiener and Szeged indices, we refer to \cite{Do.RI,Kh.P,zhang.H,ZhouB.X}.

Let $\mathcal{U}_{n, d}$ be the set of unicyclic graphs of order $n$ with diameter $d$ for $2 \leq d \leq n-2$.
Tan \cite{Tan} and Shi \cite{Shi} independently determined the graph in $\mathcal{U}_{n, d}$ with minimum Wiener index.
Liu et al. \cite{Yu} characterized the graph in $\mathcal{U}_{n, d}$ with minimum Szeged index.
The graph in $\mathcal{U}_{n, d}$ with minimum edge Szeged index was identified by Wang et al. \cite{AMCLSC}.
Yu et al. \cite{YUERZ} characterized the graph with minimum revised Szeged index among all the unicyclic graphs with given order and diameter.
In this paper, the graphs in $\mathcal{U}_{n, d}$ with  minimum revised edge Szeged index are characterized. In order to state our results, we need the following notations.

Denote by $P_n$, $S_n$ and $C_n$ a path, star and cycle on $n$ vertices, respectively. We define the root vertex of a star is its center vertex and the root vertex of a path is its one pendant vertex if there is no further explanation.
Let $k_{1}$, $k_{2}$ and $i$ be three nonnegative integers,  $P' =P_{k_{1}+1}$ be a path with a terminal vertex $u$ and $P'' =P_{k_{2}+1}$ be a path with a terminal vertex $v$. Let $S'=S_{i+1}$ be a star with center vertex $w$. Denote by $P^{i}_{k_{1},k_{2}}$ the tree obtained
from $P'$, $P''$ and $S'$ by identifying $u$, $v$ and $w$ to $u'$, and call $u'$ the root vertex of $P^{i}_{k_{1},k_{2}}$ (see Fig. 1). For
convenience, write $P^{i}_{k_{1}}$ for $P^{i}_{k_{1},0}$ (see Fig. 1). It can be checked that $P^{i}_{0,0}\cong S_{i+1}$.

\begin{center}   \setlength{\unitlength}{0.7mm}
\begin{picture}(30,40)

\put(-10,25){\circle*{1.5}}
\put(-20,20){\circle*{1.5}}
\put(-20,30){\circle*{1.5}}
\put(-25,17.5){\circle*{1.5}}
\put(-25,32.5){\circle*{1.5}}
\put(-35,12.5){\circle*{1.5}}
\put(-35,37.5){\circle*{1.5}}
\put(-22.5,18.75){\circle*{1}}
\put(-22.5,31.25){\circle*{1}}
\put(-10,25){\line(-2,-1){10}}
\put(-10,25){\line(-2,1){10}}

\put(0,20){\circle*{1.5}}
\put(0,30){\circle*{1.5}}
\put(0,23){\circle*{1}}
\put(0,25){\circle*{1}}
\put(0,27){\circle*{1}}

\put(-10,25){\line(2,-1){10}}
\put(-10,25){\line(2,1){10}}

\put(-25,32.5){\line(-2,1){10}}
\put(-25,17.5){\line(-2,-1){10}}
\put(-50,12.5){\scriptsize$P_{k_{1}+1}$}
\put(-50,37.5){\scriptsize$P_{k_{2}+1}$}

\put(-13,19){\scriptsize$u'$}
\put(1,25){\scriptsize$i$}

\put(60,25){\circle*{1.5}}

\put(50,30){\circle*{1.5}}
\put(45,32.5){\circle*{1.5}}
\put(35,37.5){\circle*{1.5}}
\put(47.5,31.25){\circle*{1}}
\put(60,25){\line(-2,1){10}}

\put(70,20){\circle*{1.5}}
\put(70,30){\circle*{1.5}}
\put(70,23){\circle*{1}}
\put(70,25){\circle*{1}}
\put(70,27){\circle*{1}}

\put(60,25){\line(2,-1){10}}
\put(60,25){\line(2,1){10}}

\put(45,32.5){\line(-2,1){10}}

\put(20,37.5){\scriptsize$P_{k_{1}+1}$}

\put(57,19){\scriptsize$u'$}
\put(71,25){\scriptsize$i$}

\put(55,10){\scriptsize$P_{k_{1}}^{i}$}
\put(-18,10){\scriptsize$P_{k_{1},k_{2}}^{i}$}
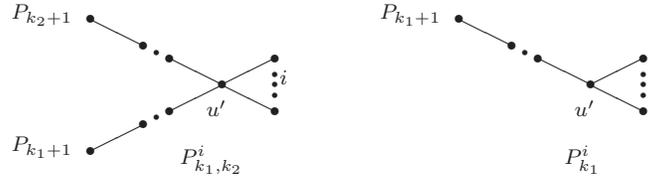
\put(-15,0){\scriptsize  Fig. 1. Graphs $P_{k_{1},k_{2}}^{i}$ and $P_{k_{1}}^{i}$}

\end{picture} \end{center}

Let $l \geq 3$ be an integer and $C_{l}=v_{1}v_{2} \cdots v_{l}v_{1}$ be a cycle with length $l$.
Let $T_{i}$ $(1 \leq i \leq l)$ be a tree with the root vertex of $T_{i}$ be $u_{i}$.
Denote by $C_{T_{1}, T_{2}, \cdots, T_{l}}^{u_{1}, u_{2}, \cdots, u_{l}}$ the unicyclic graph obtained from $C_{l}$ by identifying the root vertex $u_{i}$ of $T_{i}$ with $v_{i}$ for each $1 \leq i \leq l$.
If there is no confusion, we write $C_{l}(T_{1}, T_{2}, \cdots, T_{l})$ for $C_{T_{1}, T_{2}, \cdots, T_{l}}^{u_{1}, u_{2}, \cdots, u_{l}}$.
Then, any unicyclic graph $G$ with a $l$-cycle is of the form $C_{l}(T_{1}, T_{2}, \cdots, T_{l})$
with $\sum_{i=1}^{l}|V(T_{i})|=|V(G)|=\sum_{i=1}^{l}|E(T_{i})|+l$.

 The graphs $G=G'$ means that $G \cong G'$. Let $G$ be a unicyclic graph of order $n$ with diameter $d$. Then $1 \leq d \leq n-2$. If $d=1$, then $G = C_{3}$. If $d=2$, then $G\in \{ C_{4}, C_{3}(P_{1}^{0}, S_{1}, S_{1})\}$ for $n=4$; $G\in \{ C_{5}, C_{3}(P_{1}^{1}, S_{1}, S_{1})\}$ for $n=5$ and $G= C_{3}(P_{1}^{n-4}, S_{1}, S_{1})$ for $n \geq 6$.
For $1 \leq d \leq 2$, it is not difficult to determine $Sz^{*}_e(G)$. But for $d > 2$, no any results about $Sz^{*}_e(G)$.
In the paper, we completely determined the minimum extremal graphs of $Sz^{*}_e(G)$ for $3 \leq d \leq n-2$.
Our main result is the following Theorem \ref{Th1}.

\begin{theorem}\label{Th1} Let $G$ be the graph with minimum revised edge Szeged index among the graphs in $\mathcal{U}_{n,d}$ with $n > 15$.\\
{\em(i)} If $d=n-2$, then $G=C_3( P_{\lceil \frac{d-1}{2} \rceil+1},P_{\lfloor \frac{d-1}{2} \rfloor+1},S_{1})$;\\
{\em(ii)} If $d=n-3$, then $G=C_3(P_{\lfloor \frac{d}{2} \rfloor ,d- \lfloor \frac{d}{2} \rfloor}^{n-d-3},S_1,S_1)$;\\
{\em(iii)} If $ 6 \leq d \leq n-4$, then $G= C_{4}(P_{\lfloor \frac{d}{2} \rfloor ,d- \lfloor \frac{d}{2} \rfloor}^{n-d-4},S_1,S_1,S_1    )$;\\
{\em(iv)} If $ 4 \leq d  \leq 5$, then $G=C_{4}(P_{d-2}^{n-d-2}, S_1,S_1,S_1  )$;\\
{\em(v)} If $d=3$, then $G=C_4(S_{n-3},S_1,S_1,S_1)$.
\end{theorem}

The rest of this paper is organized as follows. In Section 2, some useful lemmas are presented.
In Section 3, we establish some transformations of the unicyclic graphs which keep
the diameter but decrease the revised edge Szeged index, and we prove that the cycle length of the graph in $\mathcal{U}_{n,d}$ with  minimum revised edge Szeged index is 3 or 4. In Section 4, the graphs in $\mathcal{U}_{n,d}$ with cycle length 3 and minimum revised edge Szeged index are characterized.
Moreover, the graphs in $\mathcal{U}_{n,d}$ with cycle length 4 and minimum revised edge Szeged index are identified in Section 5.
In Section 6, Theorem \ref{Th1} is proved.

\section{Some useful lemmas}
In this section, we introduce some lemmas that will be used later. Firstly, we introduce some notations.
For $v \in V(G)$, we define $$D(v|G)=\sum_{w\in V(G)}d(v, w|G).$$ For an integer $g$, define $$ \delta(g)=\left\{
                                                 \begin{array}{ll}
                                                   1, & \hbox{if $g$ is odd ;} \\
                                                   0, & \hbox{if $g$ is even.}
                                                 \end{array}
                                               \right.
  $$
For any edge $e=xy \in E(G)$, define
$$m(e|G)=m_{x}(e|G)m_{y}(e|G)$$
and
$$m_{e}^{*}(e|G)=[m_{x}(e|G)+\frac{m_{0}(e|G)}{2}][m_{y}(e|G)+\frac{m_{0}(e|G)}{2}].$$

\begin{lemma}\label{n1} Let $G=C_{g}(T_{1}, T_{2}, \cdots, T_{g})$ with $|V(G)|=n$. Then

$$Sz^*_{e}(G)=Sz_{e}(G)+\frac{1}{4}n(2n-1)+\frac{1}{4}(2n-3)g+\delta(g)[\frac{1}{4}g(5-4n)+\frac{n^{2}-n}{2}-\frac{1}{4}\sum_{i=1}^g |E(T_{i})|^{2}]$$
where
\[
\delta(g)= \left\{ \begin{array}{ll} 0, & \mbox{ if } g \mbox{ is even,}
\\
1, & \mbox{ if } g \mbox{ is odd.}
\end{array}\right.
\]

\end{lemma}

\begin{proof}We divide the edges of $G$ into two types:

(a) the edges belonging to the tree $T_{i}$ for $i=1, 2, \cdots, g$;

(b) the edges belonging to the cycle $C_{g}$.

Firstly, we consider the edges of type (a). For each edge $e=xy$ of $T_{i}$ $(i \in [1, g])$.
It can be checked that $m_{x}(e|G)+m_{y}(e|G)=n-1$ and $m_{0}(e|G)=1$.
Let $\mu$ be the contributions to $Sz^*_{e}(G)$ of the edges of type (a), then
\begin{eqnarray*}
\mu&=&\sum_{i=1}^g \sum\limits_{e=xy\in E(T_{i})}[m_{x}(e|G)+\frac{m_{0}(e|G)}{2}][m_{y}(e|G)+\frac{m_{0}(e|G)}{2}]\\
&=&\sum_{i=1}^g \sum\limits_{e=xy\in E(T_{i})}m(e|G)+\sum_{i=1}^g\sum\limits_{e=xy\in E(T_{i})}[\frac{m_{0}(e|G)}{2}(m_{x}(e|G)+m_{y}(e|G))+\frac{(m_{0}(e|G))^{2}}{4}]\\
&=&\sum_{i=1}^g \sum\limits_{e=xy\in E(T_{i})}m(e|G)+\sum_{i=1}^g\sum\limits_{e=xy\in E(T_{i})}[\frac{1}{2}(n-1)+\frac{1}{4}]\\
&=&\sum_{i=1}^g \sum\limits_{e=xy\in E(T_{i})}m(e|G)+(n-g)[\frac{1}{2}(n-1)+\frac{1}{4}]\\
&=&\sum_{i=1}^g \sum\limits_{e=xy\in E(T_{i})}m(e|G)+\frac{1}{4}(2n-1)(n-g).
\end{eqnarray*}

Now we consider the edges of type (b). We divided into two cases according to the parity of $g$.

{\bf  Case 1.}  $g$ is even.

For each edge $e=xy \in E(C_{g})$, it can be checked that $m_{x}(e|G)+m_{y}(e|G)=n-2$ and $m_{0}(e|G)=2$. Let $\lambda_{1}$ be the contributions to $Sz^*_{e}(G)$ of the edges of type (b). Then
\begin{eqnarray*}
\lambda_{1}&=&\sum\limits_{e=xy\in E(C_{g})}[m_{x}(e|G)+\frac{m_{0}(e|G)}{2}][m_{y}(e|G)+\frac{m_{0}(e|G)}{2}]\\
&=&\sum\limits_{e=xy\in E(C_{g})}m_{x}(e|G)m_{y}(e|G)+\sum\limits_{e=xy\in E(C_{g})}[\frac{m_{0}(e|G)}{2}(m_{x}(e|G)+m_{y}(e|G))+\frac{m^{2}_{0}(e|G)}{4}]\\
&=&\sum\limits_{e=xy\in E(C_{g})}m_{x}(e|G)m_{y}(e|G)+\sum\limits_{i=1}^{g}[\frac{2}{2}(n-2)+\frac{4}{4}]\\
&=&\sum\limits_{e=xy\in E(C_{g})}m_{x}(e|G)m_{y}(e|G)+g(n-1).
\end{eqnarray*}

By the definition of revised edge Szeged index, we have
\begin{eqnarray*}
Sz^*_{e}(G)&=&\mu+\lambda_{1}\\
&=&\sum_{i=1}^g \sum\limits_{e=xy\in E(T_{i})}m(e|G)+\frac{1}{4}(2n-1)(n-g)+\sum\limits_{e=xy\in E(C_{g})}m(e|G)+g(n-1)\\
&=&Sz_{e}(G)+\frac{1}{4}(2n-1)n+\frac{1}{4}(2n-3)g.
\end{eqnarray*}

{\bf  Case 2.}  $g$ is odd.

Let $\lambda_{2}$ be the contributions to $Sz^*_{e}(G)$ of the edges of type (b). It can be checked that
\begin{eqnarray*}
\lambda_{2}&=&\sum\limits_{e=xy\in E(C_{g})}[m_{x}(e|G)+\frac{m_{0}(e|G)}{2}][m_{y}(e|G)+\frac{m_{0}(e|G)}{2}]\\
&=&\sum\limits_{e=xy\in E(C_{g})}m_{x}(e|G)m_{y}(e|G)+\sum\limits_{e=xy\in E(C_{g})}[\frac{m_{0}(e|G)}{2}(m_{x}(e|G)+m_{y}(e|G))+\frac{m^{2}_{0}(e|G)}{4}]\\
&=&\sum\limits_{e=xy\in E(C_{g})}m_{x}(e|G)m_{y}(e|G)+ \sum\limits_{i=1}^{g}[\frac{|E(T_{i})|+1)|}{2}(n-|E(T_{i})|-1)+\frac{(|E(T_{i})|+1)^{2}}{4}]\\
&=&\sum\limits_{e=xy\in E(C_{g})}m_{x}(e|G)m_{y}(e|G)+\frac{n^{2}}{2}-\sum\limits_{i=1}^{g}\frac{|E(T_{i})|^{2}}{4}-\sum\limits_{i=1}^{g}\frac{|E(T_{i})|}{2}-\frac{1}{4}g\\
&=&\sum\limits_{e=xy\in E(C_{g})}m_{x}(e|G)m_{y}(e|G)+\frac{n^{2}}{2}-\sum\limits_{i=1}^{g}\frac{|E(T_{i})|^{2}}{4}-\frac{n-g}{2}-\frac{1}{4}g.
\end{eqnarray*}

By the definition of revised edge Szeged index, we have
\begin{eqnarray*}
Sz^*_{e}(G)&=&\mu+\lambda_{2}\\
&=&Sz_{e}(G)+\frac{1}{4}(4n-3)n-\frac{1}{2}(n-1)g-\sum\limits_{i=1}^{g}\frac{|E(T_{i})|^{2}}{4}.
\end{eqnarray*}

The proof is completed.
\end{proof}

\begin{lemma}\label{01}{\em(\cite{LJP})}
Let $G=C_r(T_1,T_2, \cdots,T_r)$ with $|V(G)|=n$. Then
\[
Sz_{e}(G)=Sz(G)+\sum_{i=1}^r D(v_{i}|T_{i})-n^{2}+ \left\{ \begin{array}{ll} nr, & \mbox{ if } r \mbox{ is odd,}
\\
r & \mbox{ if } r \mbox{ is even.}
\end{array}\right.
\]
\end{lemma}

From Lemmas \ref{n1} and \ref{01}, we have the following Corollary \ref{02}.

\begin{corollary}\label{02} Let $G=C_{g}(T_{1}, T_{2}, \cdots, T_{g})$ be a unicyclic graph with $|V(G)|=n$. Then
\[
Sz^*_{e}(G)=Sz(G)+\sum_{i=1}^g D(v_{i}|T_{i})-\frac{1}{4}(2n+1)(n-g)+\delta(g)[\frac{1}{4}g+\frac{n^{2}-n}{2}-\frac{1}{4}\sum_{i=1}^g |E(T_{i})|^{2}]
\]
where
\[
\delta(g)= \left\{ \begin{array}{ll} 0, & \mbox{ if } g \mbox{ is even,}
\\
1, & \mbox{ if } g \mbox{ is odd.}
\end{array}\right.
\]

\end{corollary}

\begin{lemma}\label{001}{\em(\cite{LMM})} Let $G$ be a graph with $|E(G)|=m$. Then
$$Sz^*_{e}(G)=\frac{m^{3}}{4}-\frac{1}{4}\sum\limits_{e=xy\in E(G)}[m_{x}(e|G)-m_{y}(e|G)]^{2}.$$
\end{lemma}

\begin{lemma}\label{lem2.8}{\rm \cite{HSJDAM}}
 Let $C_{g}=v_{1}v_{2} \cdots  v_{g}v_{1}$ and $G=C_{g}(T_{1}, T_{2}, \cdots, T_{g})$ with $|V(G)|=n$. Let
$$S_1=\sum\limits_{uv \in E(G)\setminus E(C_{g})}m_{u}(uv|G)m_{v}(uv|G) \mbox{ and } S_2=\sum\limits_{uv \in  E(C_{g})}m_{u}(uv|G)m_{v}(uv|G).$$
Then
\begin{eqnarray*}
S_1&=&\sum\limits_{i=1}^{g}W_{e}(T_{i})+\sum\limits_{i=1}^{g}(n-|E(T_{i})|)D(v_{i}|T_{i})
-\sum\limits_{i=1}^{g}|E(T_{i})|(n-|E(T_{i})|),\\
S_2&=&g( \lceil \frac{g-2}{2} \rceil)^{2}+\lceil \frac{g-2}{2} \rceil g(n-g)-\delta(g)\lceil \frac{g-2}{2} \rceil(n-g)
\\&&+\sum\limits_{i=1}^{g}\sum\limits_{j=1}^{g}|E(T_{i})||E(T_{j})|d(v_{i}, v_{j}|C_{g})-\delta(g)\sum\limits_{i < j}|E(T_{i})||E(T_{j})|.
\end{eqnarray*}
\end{lemma}

\begin{lemma}\label{lem2.3}{\rm \cite{S.H.H.Y}}
Let $G$ and $G'$ be the graphs shown as in Fig. 2, where $G$ consists of $G_0$ and $G_1$ with a common vertex $u$, and $G'$ consists of $G_0$ and $G_2$ with a common vertex $u$. Then each of the followings holds:\\
{\em(i)} For any edge $e=w_1w_2\in E(G_0)$ and $1\leq i\leq 2$, we have
                 $$m_{w_i}(e|G)=m_{w_i}(e|G_0)+\tau(u)|E(G_1)|,$$
                 where $$ \tau(u)=\left\{
                                                 \begin{array}{ll}
                                                   1, & \hbox{$u\in N_{w_i}(e|G_0)$ ;} \\
                                                   0, & \hbox{otherwise.}
                                                 \end{array}
                                               \right.
  $$
{\em(ii)} If $|E(G_1)|=|E(G_2)|$, then
\begin{eqnarray*}
\sum\limits_{e=w_1w_2 \in E(G_0)} m_{w_1}(e|G)m_{w_2}(e|G)&=&\sum\limits_{e=w_1w_2 \in E(G_0)} m_{w_1}(e|G')m_{w_2}(e|G').
\end{eqnarray*}
\end{lemma}

\begin{center}   \setlength{\unitlength}{0.7mm}
\begin{picture}(30,40)

\put(-15,30){\circle*{1}}
\put(-25,30){\circle{20}}
\put(-5,30){\circle{20}}
\put(-18,5){\scriptsize$G$}
\put(42,5){\scriptsize$G'$}

\put(-21,30){\scriptsize$u$}
\put(39,30){\scriptsize$u$}

\put(45,30){\circle*{1}}

\put(-27,15){\scriptsize$G_{0}$}
\put(-7,15){\scriptsize$G_{1}$}
\put(33,15){\scriptsize$G_{0}$}
\put(53,15){\scriptsize$G_{2}$}

\put(35,30){\circle{20}}
\put(55,30){\circle{20}}

\put(-15,-2){\scriptsize Fig. 2. $G$ and $G'$ in Lemma \ref{lem2.3}}
\end{picture} \end{center}

\begin{lemma}\label{lem2.5}{\rm \cite{S.H.H.Y}} Let $G$ be a graph of order $n$ with a cycle $C_g=v_1v_2 \cdots v_{g}v_1$.
Assume that $G-E(C_{g})$ has exactly $g$ components $G_1,G_2,\cdots, G_g$, where $G_i$ is the component of $G-E(C_{g})$ that contains $v_i$ for $1\leq i\leq g$.
Let
 \begin{eqnarray*}G'=G- \displaystyle\cup_{i=2}^g\{wv_i: w\in N_{G_i}(v_i)\}+ \cup_{i=2}^g\{wv_1:  w\in N_{G_i}(v_i)\}.\end{eqnarray*}
Then $Sz_e(G') \leq Sz_e(G)$ with equality if and only if $C_{g}$ is an end-block, that is, $G\cong G'$.
\end{lemma}



Let $n\geq 3$ and $2 \leq d \leq n-1$ be integers. Let $P_{d+1}=u_{0}u_{1} \cdots u_{d}$ be a path of order $d+1$ and $T_{n, d, \lfloor \frac{d}{2} \rfloor }$
be the tree obtained from $P_{d+1}$ by attaching $n-d-1$ pendent vertices to $u_{\lfloor \frac{d}{2} \rfloor }$.
Note that $T_{n, d, \lfloor \frac{d}{2} \rfloor } \cong P^{n-d-1}_{\lfloor \frac{d}{2} \rfloor , \lceil \frac{d}{2} \rceil }$.

\begin{lemma}\label{lem2.6}{\rm \cite{HQLIU}}
Among the trees of order $n$ with diameter $2 \leq d\leq n-2$, $T_{n, d, \lfloor \frac{d}{2} \rfloor }$ has minimum Wiener index.
\end{lemma}

\begin{lemma}\label{lem2.60}{\rm \cite{BKLY}}
Let $T$ be a tree of order $n$. Then
$$W_{e}(T)=W(T)-\frac{n^{2}-n}{2}.$$
\end{lemma}

\begin{lemma}\label{lem2.9}{\rm \cite{Yu}} Let $C_g=v_1v_2 \cdots v_{g}v_1$ and $d_{i, j}=d(v_{i}, v_{j}|C_g)$ for $i, j \in \{1, 2, \cdots, g \}$.
\\
{\em(i)} If $g$ is an even number, then
$$  (d_{2, j}-d_{1, j}+1, d_{g, j}-d_{1, j}+1) =     \left\{
  \begin{array}{ll}
    (0, 2), & \hbox{if $2 \leq j \leq \frac{g}{2}$;} \\
    (0, 0), & \hbox{if $j=\frac{g}{2}+1$;} \\
    (2, 0), & \hbox{if $\frac{g}{2}+2 \leq j \leq g-1.$}
  \end{array}
\right.
$$
{\em(ii)} If $g$ is an odd number, then
$$      (d_{2, j}-d_{1, j}+1, d_{g, j}-d_{1, j}+1) =\left\{
  \begin{array}{ll}
    (0, 2), & \hbox{if $2 \leq j \leq \frac{g-1}{2}$;} \\
    (0, 1), & \hbox{if $j=\frac{g+1}{2}$;} \\
    (1, 0), & \hbox{if $j=\frac{g+3}{2}$;} \\
    (2, 0), & \hbox{if $\frac{g+5}{2} \leq j \leq g-1.$}
  \end{array}
\right.
$$
\end{lemma}

\section{Some useful transformations}

In this section, some transformations of unicyclic graphs which decrease the revised edge Szeged index of the graphs are presented.

Let $H_{i} $ $(0 \leq i \leq l )$ be a tree with the root vertex $u_{i}'$ and $P=u_{0}u_{1}u_{2} \cdots u_{l}$ be a path on $l+1$ vertices.
Denote by $P_{H_{0},H_{1},\cdots,H_{l}}^{u_{0}',u_{1}',\cdots,u_{l}'}$ the tree obtained from $P$ by identifying
the root vertex $u_{i}'$ of $H_{i}$ with $u_{i}$ for each $0 \leq i \leq l$. If there is no confusion, we write $P(H_{0},H_{1},\cdots,H_{l})$ for $P_{H_{0},H_{1},\cdots,H_{l}}^{u_{0}',u_{1}',\cdots,u_{l}'}$. For each $0 \leq i \leq l$, if $H_{i}$ is a star of order $a_{i}+1$ and the root vertex of $H_{i}$
is the center vertex of $H_{i}$, then we write $P(a_{0},a_{1},\cdots,a_{l})$ for $P(H_{0},H_{1},\cdots,H_{l})$.
It is well known that $P(a_{0},a_{1},\cdots,a_{l})$ is a $caterpillar$ $graph$ and $P$ is called the $backbone$ of $P(a_{0},a_{1},\cdots,a_{l})$.

\begin{lemma}\label{lem3.1} Let $g \geq 3$ be an integer and $C_{g}=v_{1}v_{2} \cdots v_{g}v_{1} $ be a cycle. Let $G=C_{g}(T_{1}, T_{2}, \cdots, T_{g})$ and
$G'=C_{g}(T'_{1}, T'_{2}, \cdots, T'_{g})$, where $|V(T_{1})|=|V(T'_{1})|=n_{1}$ and $T_{i}=T'_{i}$ for $2 \leq i \leq g$.\\
{\em(i)} Let $T'_{1} \cong S_{n_{1}} $ and the root vertex of $T'_{1}$ is the center vertex of the star $S_{n_{1}}$. Then $Sz_{e}(G) \geq Sz_{e}(G')$ and $Sz_{e}^{*}(G) \geq Sz_{e}^{*}(G')$ with the equalities hold if and only if
$G\cong G'$;
\\{\em(ii)} Let $2 \leq d \leq n-2$ and $T'_{1} \cong T_{n_{1}, d, \lfloor \frac{d}{2} \rfloor }$ be the tree with root vertex $u_{\lfloor \frac{d}{2} \rfloor}$. If the diameter of $T_{1}$ is $d$, then $Sz_{e}(G) \geq Sz_{e}(G')$ and $Sz_{e}^{*}(G) \geq Sz_{e}^{*}(G')$  with the equalities hold if and only if
$G\cong G'$;
\\{\em(iii)} Let $T_{1}$ be a tree with root vertex $v_{1}$ such that $l = {\rm{max}}_{v \in V(T_{1})} d(v, v_{1}| T_{1}) $. Let $P=u_{0}u_{1}u_{2} \cdots u_{l}\, \\(u_{l}=v_{1})$ be a path and $T_{1} \cong P(S_{1}, H_{1}, H_{2}, \cdots , H_{l})$. Let $T'_{1} \cong P(0, |V(H_{1})|-1, |V(H_{2})|-1, \cdots , |V(H_{l})|-1)$ and let the root vertex of $T'_{1}$ be $u_{l}$. Then $Sz_{e}(G) \geq Sz_{e}(G')$ and $Sz_{e}^{*}(G) \geq Sz_{e}^{*}(G')$ with the equalities hold if and only if $G\cong G'$.

\end{lemma}

\begin{proof}
By Lemma~\ref{lem2.3},
\begin{eqnarray}Sz_{e}(G) - Sz_{e}(G')=\sum\limits_{e\in E(T_{1})}m(e|G)-\sum\limits_{e\in E(T'_{1})}m(e|G').\label{E-01}\end{eqnarray}
From Lemma \ref{n1}, one has
\begin{eqnarray}Sz_{e}^{*}(G) - Sz_{e}^{*}(G')=Sz_{e}(G) - Sz_{e}(G').\label{E-011}\end{eqnarray}

(i) As $T'_{1} = S_{n_{1}} $, $\sum\limits_{e\in E(T_{1})}m(e|G)\geq 0=\sum\limits_{e\in E(T'_{1})}m(e|G')$ with equality if and only if $T_{1}=T'_{1}$. From (\ref{E-01}) and Lemma \ref{n1}, (i) holds immediately.

(ii) From Lemmas \ref{lem2.8} and \ref{lem2.60}, and the fact that $|E(G)|=|E(G')|$ and $|E(T_1)|=|E(T_1')|$, we have
 \begin{eqnarray}Sz_{e}(G)-Sz_{e}(G')=W(T_1)-W(T'_1)+
(|E(G)|-|E(T_1)|)\big(D(v_1|T_1)-D(v_1|T'_1)\big).\label{E-02}\end{eqnarray}

Since the diameter of $T_1$ is $d$, by Lemma \ref{lem2.6}, \begin{eqnarray}W(T_1)\geq W(T'_1)\label{E-03},\end{eqnarray} and the equality holds if and only if $T_1\cong T'_1$. Let $P_{d+1}=u_0u_1\cdots u_d$ be the path of length $d$ in $T_1$. Then \begin{eqnarray}D(v_1|T_1)&=&\sum_{v\not\in V(P_{d+1})}d(v_1,v|T_1)+\sum_{v\in V(P_{d+1})}d(v_1,v|T_1)\nonumber\\
&\geq& |V(T_1)|-d-1+\sum_{v\in V(P_{d+1})}d(u_{\lfloor\frac{d}{2}\rfloor},v|P_{d+1})\nonumber\\
&=& D(u_{\lfloor\frac{d}{2}\rfloor}|T'_1),\label{E-04}\end{eqnarray} and the equality holds if and only if $T_1\cong T'_1$ and $v_1=u_{\lfloor\frac{d}{2}\rfloor}$. By (\ref{E-02})-(\ref{E-04}) and Lemma \ref{n1}, we have $Sz_{e}(G) \geq Sz_{e}(G')$ and $Sz_{e}^{*}(G) \geq Sz_{e}^{*}(G')$  with the equalities hold if and only if
$G\cong G'$.

(iii) As pendent edges make no contributions to the edge Szeged index of a graph and
$$\sum\limits_{e\in E(P)}m(e|G)=\sum\limits_{e\in E(P)}m(e|G').$$
By (\ref{E-01})-(\ref{E-011})  and Lemma \ref{n1}, (iii) holds immediately.
\end{proof}

\begin{definition}\label{def3.2} Let $g \geq 3$ be an integer and $C_{g}=v_{1}v_{2} \cdots v_{g}v_{1} $ be a cycle.\\
{\em(i)} Let $G=C_{g}(T_{1}, T_{2}, \cdots, T_{g})$, where $T_{1} \cong P(0, a_{1}, \cdots, a_{l})$ is a caterpillar with backbone $P=u_{0}u_{1}u_{2} \cdots u_{l}$, and the root vertex of $T_{1}$ is $u_{l}=v_{1}$. Denote by $X_{k}-1$ the edge size of the connected component of $G-u_{k}-u_{k+1}$ containing $u_{k-1}$,
and let $Y_{k}=|E(G)|-X_{k}-a_{k}-a_{k+1}-1$.\\
{\em(ii)} For $1 \leq k \leq l-1$, let $G'$ be the graph obtained from $G$ by deleting edges in $\{u_{k}v:v \in N_{T_{1}}(u_{k}) \setminus V(P) \}$ and adding edges
in $\{u_{k+1}v:v \in N_{T_{1}}(u_{k}) \setminus V(P) \}$. Let $G''$ be the graph obtained from $G$ by deleting edges in set $\{u_{k+1}v:v \in N_{T_{1}}(u_{k+1}) \setminus V(P) \}$ and adding edges in set $\{u_{k}v:v \in N_{T_{1}}(u_{k+1}) \setminus V(P) \}$.

\end{definition}

\begin{lemma}\label{lem3.3}

Let $G$, $G'$ and $G''$ be the graphs defined in Definition 3.2. If $a_{k} > 0$ and $Y_{k}+a_{k+1} > X _{k}$, then $Sz_{e}(G') < Sz_{e}(G)$ and $Sz_{e}^{*}(G') < Sz_{e}^{*}(G)$. If $a_{k+1} > 0$ and
$X_{k}+a_{k} > Y_{k}$, then $Sz_{e}(G'') < Sz_{e}(G)$ and $Sz_{e}^{*}(G'') < Sz_{e}^{*}(G)$.
\end{lemma}

\begin{proof}
Note that
$$\sum \limits_{ \substack {e \in E(G) \\  e \neq u_{k}u_{k+1} }} m(e|G)= \sum \limits_{ \substack {e \in E(G') \\  e \neq u_{k}u_{k+1} }} m(e|G')=\sum \limits_{ \substack {e \in E(G'') \\  e \neq u_{k}u_{k+1} }} m(e|G'').$$
By Lemma~\ref{lem2.3}, we have
\begin{eqnarray*}Sz_{e}(G) - Sz_{e}(G')&=&m(u_{k}u_{k+1}|G)-m(u_{k}u_{k+1}|G')=a_{k}(Y_{k}+a_{k+1}-X_{k}),\\
Sz_{e}(G) - Sz_{e}(G'')&=&m(u_{k}u_{k+1}|G)-m(u_{k}u_{k+1}|G'')=a_{k+1}(X_{k}+a_{k}-Y_{k}).\end{eqnarray*}

From Lemma \ref{n1}, we have
$Sz_{e}^{*}(G) - Sz_{e}^{*}(G')=Sz_{e}(G) - Sz_{e}(G')$ and $Sz_{e}^{*}(G) - Sz_{e}^{*}(G'')=Sz_{e}(G) - Sz_{e}(G'')$.
Thus, Lemma~\ref{lem3.3} holds immediately.
\end{proof}

\begin{center}   \setlength{\unitlength}{0.7mm}
\begin{picture}(30,50)

\put(-50,10){\circle*{1.5}}
\put(-50,20){\circle*{1.5}}
\put(-20,10){\circle*{1.5}}
\put(-20,20){\circle*{1.5}}

\put(-50,10){\line(0,1){10}}
\put(-20,10){\line(0,1){10}}

\put(-50,10){\line(1,0){10}}
\put(-20,10){\line(-1,0){10}}
\put(-50,20){\line(1,0){10}}
\put(-20,20){\line(-1,0){10}}

\put(-60,20){\circle*{1.5}}
\put(-65,20){\circle*{1.5}}
\put(-75,20){\circle*{1.5}}

\put(-75,20){\line(1,0){10}}
\put(-60,20){\line(1,0){10}}
\put(-62.5,20){\circle*{1}}

\put(-62.5,32.5){\circle*{1}}

\put(-60,30){\circle*{1.5}}
\put(-65,35){\circle*{1.5}}
\put(-75,45){\circle*{1.5}}

\put(-75,45){\line(1,-1){10}}
\put(-60,30){\line(1,-1){10}}

\put(-35,20){\circle*{1}}
\put(-37,20){\circle*{1}}
\put(-33,20){\circle*{1}}

\put(-35,10){\circle*{1}}
\put(-37,10){\circle*{1}}
\put(-33,10){\circle*{1}}

\put(-45,30){\circle*{1}}
\put(-47,30){\circle*{1}}
\put(-43,30){\circle*{1}}

\put(-25,30){\circle*{1}}
\put(-27,30){\circle*{1}}
\put(-23,30){\circle*{1}}
\put(-50,30){\circle*{1.5}}
\put(-40,30){\circle*{1.5}}
\put(-20,30){\circle*{1.5}}
\put(-30,30){\circle*{1.5}}

\put(-50,30){\line(0,-1){10}}
\put(-40,30){\line(-1,-1){10}}
\put(-20,30){\line(0,-1){10}}
\put(-30,30){\line(1,-1){10}}

\put(-10,20){\circle*{1.5}}
\put(-5,20){\circle*{1.5}}
\put(5,20){\circle*{1.5}}

\put(-5,20){\line(1,0){10}}
\put(-20,20){\line(1,0){10}}
\put(-7.5,20){\circle*{1}}
\put(-62.5,32.5){\circle*{1}}

\put(-10,30){\circle*{1.5}}
\put(-5,35){\circle*{1.5}}
\put(5,45){\circle*{1.5}}

\put(5,45){\line(-1,-1){10}}
\put(-10,30){\line(-1,-1){10}}

\put(-50,17){\scriptsize$v_{k}$}
\put(50,17){\scriptsize$v_{k}$}

\put(-24,17){\scriptsize$v_{l}$}

\put(76,17){\scriptsize$v_{l}$}


\put(50,10){\circle*{1.5}}
\put(50,20){\circle*{1.5}}
\put(80,10){\circle*{1.5}}
\put(80,20){\circle*{1.5}}

\put(50,10){\line(0,1){10}}
\put(80,10){\line(0,1){10}}

\put(50,10){\line(1,0){10}}
\put(80,10){\line(-1,0){10}}
\put(50,20){\line(1,0){10}}
\put(80,20){\line(-1,0){10}}

\put(40,20){\circle*{1.5}}
\put(35,20){\circle*{1.5}}
\put(25,20){\circle*{1.5}}

\put(25,20){\line(1,0){10}}
\put(40,20){\line(1,0){10}}
\put(37.5,20){\circle*{1}}

\put(37.5,32.5){\circle*{1}}

\put(40,30){\circle*{1.5}}
\put(35,35){\circle*{1.5}}
\put(25,45){\circle*{1.5}}

\put(25,45){\line(1,-1){10}}
\put(40,30){\line(1,-1){10}}

\put(65,20){\circle*{1}}
\put(63,20){\circle*{1}}
\put(67,20){\circle*{1}}

\put(65,10){\circle*{1}}
\put(63,10){\circle*{1}}
\put(67,10){\circle*{1}}

\put(55,30){\circle*{1}}
\put(53,30){\circle*{1}}
\put(57,30){\circle*{1}}

\put(50,30){\circle*{1.5}}
\put(60,30){\circle*{1.5}}

\put(50,30){\line(0,-1){10}}
\put(60,30){\line(-1,-1){10}}

\put(90,20){\circle*{1.5}}
\put(95,20){\circle*{1.5}}
\put(105,20){\circle*{1.5}}

\put(95,20){\line(1,0){10}}
\put(80,20){\line(1,0){10}}
\put(92.5,20){\circle*{1}}
\put(37.5,32.5){\circle*{1}}

\put(90,30){\circle*{1.5}}
\put(95,35){\circle*{1.5}}
\put(105,45){\circle*{1.5}}

\put(105,45){\line(-1,-1){10}}
\put(90,30){\line(-1,-1){10}}

\put(-80,16){\scriptsize$P_{k_{1}+1}$}
\put(-73,44){\scriptsize$P_{k_{2}+1}$}
\put(20,16){\scriptsize$P_{k_{1}+1}$}
\put(27,44){\scriptsize$P_{k_{2}+1}$}

\put(0,16){\scriptsize$P_{l_{1}+1}$}
\put(-8,44){\scriptsize$P_{l_{2}+1}$}
\put(100,16){\scriptsize$P_{l_{1}+1}$}
\put(92,44){\scriptsize$P_{l_{2}+1}$}

\put(-46,32){\scriptsize$a$}
\put(-26,32){\scriptsize$b$}

\put(51,32){\scriptsize$a+b$}

\put(-37,4.5){\scriptsize$G$}
\put(63,4.5){\scriptsize$G'$}

\put(-25,0){\scriptsize Fig. 3. Graphs $G$ and $G'$ in Lemma \ref{09}}

\end{picture} \end{center}

\begin{lemma}\label{09}  Let $g\geq 3$, $1\leq k,l\leq g $ $(k\not=l)$, $a,b>0$ and $k_1,k_2,l_1,l_2\geq 0$ be integers.
Let $C_g=v_1v_2\cdots v_gv_1$ be a cycle with $g \geq 3$. Let $G=C_g(T_1,T_2, \cdots,T_g)$, $G'=C_g(T'_1,T'_2, \cdots,T'_g)$ and $G''=C_g(T''_1,T''_2, \cdots,T''_g) $ be three unicyclic graphs, where
$T_k=P_{k_1,k_2}^a$, $T_l=P_{l_1,l_2}^b$, $T'_k=P_{k_1,k_2}^{a+b}$, $T'_l=P_{l_1,l_2}^0$, $T''_k=P_{k_1,k_2}^{0}$, $T''_l=P_{l_1,l_2}^{a+b}$,
 and $T_i=T'_i=T''_i$ for any $i\not=k,l$ $(G$ and $G'$ are depicted in Fig. 3). For any $1\leq i\leq g$, denote $N_i=\sum_ {j\ne {i}}|V(T_j)|d(v_i,v_j|C_g)$.
If $$2N_l+\frac{1}{2}\delta(g)|V(T_l)|\ge 2N_k+\frac{1}{2}\delta(g)|V(T_k)|,$$ then $Sz^*_{e}(G)> Sz^*_{e}(G')$;  Otherwise, $Sz_{e}^*(G)> Sz_{e}^*(G'').$
\end{lemma}

\begin{proof}
For convenience, denote $d_{i,j}=d(v_i,v_j|C_g)$. By Lemma \ref{02} and $|V(G)|=|V(G')|$, we have
\begin{eqnarray}Sz_{e}^*(G)-Sz_{e}^*(G')=Sz(G)-Sz(G')+\frac{1}{4}\delta(g)(|E(T'_k)|^2+|E(T'_l)|^2-|E(T_k)|^2-|E(T_l)|^2).\label{E-05}\end{eqnarray}
Since $T_k=P_{k_1,k_2}^a$, $T_l=P_{l_1,l_2}^b$, $T'_k=P_{k_1,k_2}^{a+b}$ and $T'_l=P_{l_1,l_2}^0$, we have\begin{eqnarray}|E(T'_k)|-|E(T_k)|=b \mbox{ and } |E(T_l)|-|E(T'_l)|=b.\nonumber \end{eqnarray}
It can be checked that
\begin{eqnarray}
|E(T'_k)|^2+|E(T'_l)|^2-|E(T_k)|^2-|E(T_l)|^2=2b(|V(T_k)|-|V(T_l)|+b).\label{E-06} \end{eqnarray}
By Lemma 3.4 of \cite{Yu}, we have
\begin{eqnarray}Sz(G)-Sz(G')=b^2(2d_{k,l}-\delta(g))+b(2N_l+\delta(g)|V(T_l)|)-b(2N_k+\delta(g)|V(T_k)|).\label{E-07} \end{eqnarray}
Hence by (\ref{E-05})-(\ref{E-07}),
\begin{eqnarray*}
Sz_{e}^*(G)-Sz_{e}^*(G')
&=& b^2(2d_{k,l}-\delta(g))+b(2N_l+\delta(g)|V(T_l)|)-b(2N_k+\delta(g)|V(T_k)|)\\
&+&\frac{1}{2}b\delta(g)(|V(T_k)|-|V(T_l)|+b)\\
&=&b^2(2d_{k,l}-\frac{1}{2}\delta(g))+b(2N_l+\frac{1}{2}\delta(g)|V(T_l)|)-b(2N_k+\frac{1}{2}\delta(g)|V(T_k)|).
\end{eqnarray*}

Thus, $Sz_{e}^*(G)>Sz_{e}^*(G')$ if $2N_l+\frac{1}{2}\delta(g)|V(T_l)|\ge 2N_k+\frac{1}{2}\delta(g)|V(T_k)|$.  Similarly, if $2N_l+\frac{1}{2}\delta(g)|V(T_l)|< 2N_k+\frac{1}{2}\delta(g)|V(T_k)|$, we have $Sz_{e}^*(G)>Sz_{e}^*(G'')$.
\end{proof}

\begin{center}   \setlength{\unitlength}{0.7mm}
\begin{picture}(30,110)

\put(-70,80){\circle*{1.5}}
\put(-80,75){\circle*{1.5}}
\put(-80,65){\circle*{1.5}}
\put(-60,75){\circle*{1.5}}
\put(-60,65){\circle*{1.5}}
\put(-70,80){\line(-2,-1){10}}
\put(-70,80){\line(2,-1){10}}
\put(-80,75){\line(0,-1){10}}
\put(-60,75){\line(0,-1){10}}
\put(-80,65){\line(1,0){5}}
\put(-60,65){\line(-1,0){5}}
\put(-73,64.4){$\dots$}
\put(-80,80){\circle*{1.5}}
\put(-85,80){\circle*{1.5}}
\put(-95,80){\circle*{1.5}}
\put(-80,85){\circle*{1.5}}
\put(-85,87.5){\circle*{1.5}}
\put(-95,92.5){\circle*{1.5}}
\put(-70,80){\line(-1,0){10}}
\put(-85,80){\line(-1,0){10}}
\put(-82.5,80){\circle*{1}}
\put(-82.5,86.25){\circle*{1}}
\put(-70,80){\line(-2,1){10}}
\put(-85,87.5){\line(-2,1){10}}
\put(-72,76){\scriptsize$v_{1}$}
\put(-65,72){\scriptsize$v_{2}$}
\put(-62,61.5){\scriptsize$v_{3}$}
\put(-85,61.5){\scriptsize$v_{g-1}$}
\put(-85,73){\scriptsize$v_{g}$}
\put(-97,82){\scriptsize$P_{l_{1}+1}$}
\put(-92,92){\scriptsize$P_{l_{2}+1}$}
\put(-70,90){\circle*{1.5}}
\put(-60,85){\circle*{1.5}}
\put(-67,88.5){\circle*{1}}
\put(-65,87.5){\circle*{1}}
\put(-63,86.5){\circle*{1}}
\put(-70,80){\line(0,1){10}}
\put(-70,80){\line(2,1){10}}
\put(-65,88.5){\scriptsize$a_{1}$}

\put(-60,75){\line(2,-1){10}}
\put(-60,75){\line(2,1){10}}
\put(-50,80){\circle*{1.5}}
\put(-50,70){\circle*{1.5}}
\put(-50,77){\circle*{1}}
\put(-50,75){\circle*{1}}
\put(-50,73){\circle*{1}}
\put(-49,75){\scriptsize$a_{2}$}

\put(-15,80){\circle*{1.5}}
\put(-25,75){\circle*{1.5}}
\put(-25,65){\circle*{1.5}}
\put(-5,75){\circle*{1.5}}
\put(-5,65){\circle*{1.5}}
\put(-15,80){\line(-2,-1){10}}
\put(-15,80){\line(2,-1){10}}
\put(-25,75){\line(0,-1){10}}
\put(-5,75){\line(0,-1){10}}
\put(-25,65){\line(1,0){5}}
\put(-5,65){\line(-1,0){5}}
\put(-18,64.4){$\dots$}
\put(-25,80){\circle*{1.5}}
\put(-30,80){\circle*{1.5}}
\put(-40,80){\circle*{1.5}}
\put(-25,85){\circle*{1.5}}
\put(-30,87.5){\circle*{1.5}}
\put(-40,92.5){\circle*{1.5}}
\put(-15,80){\line(-1,0){10}}
\put(-30,80){\line(-1,0){10}}
\put(-27.5,80){\circle*{1}}
\put(-27.5,86.25){\circle*{1}}
\put(-15,80){\line(-2,1){10}}
\put(-30,87.5){\line(-2,1){10}}
\put(-17,76){\scriptsize$v_{1}$}
\put(-10,72){\scriptsize$v_{3}$}
\put(-7,61.5){\scriptsize$v_{4}$}
\put(-30,61.5){\scriptsize$v_{g-2}$}
\put(-30,72){\scriptsize$v_{g-1}$}
\put(-42,82){\scriptsize$P_{l_{1}+1}$}
\put(-37,92){\scriptsize$P_{l_{2}+1}$}
\put(-15,90){\circle*{1.5}}
\put(-5,85){\circle*{1.5}}
\put(-12,88.5){\circle*{1}}
\put(-10,87.5){\circle*{1}}
\put(-8,86.5){\circle*{1}}
\put(-15,80){\line(0,1){10}}
\put(-15,80){\line(2,1){10}}
\put(-10,88.5){\scriptsize$a_{1}+a_{2}$}


\put(-5,80){\circle*{1.5}}
\put(-20,95){\circle*{1.5}}
\put(-15,80){\line(1,0){10}}
\put(-15,80){\line(-1,3){5}}
\put(-4,79){\scriptsize$v_{2}$}
\put(-20,95){\scriptsize$v_{g}$}

\put(-70,55){\scriptsize$G$}
\put(-15,55){\scriptsize$G'$}
\put(-42.5,50){\scriptsize(a)}
\put(75,50){\scriptsize(b)}
\put(50,55){\scriptsize$G$}
\put(100,55){\scriptsize$G'$}

\put(-10,30){\circle*{1.5}}
\put(-20,25){\circle*{1.5}}
\put(-20,15){\circle*{1.5}}
\put(0,25){\circle*{1.5}}
\put(0,15){\circle*{1.5}}
\put(-10,30){\line(-2,-1){10}}
\put(-10,30){\line(2,-1){10}}
\put(-20,25){\line(0,-1){10}}
\put(0,25){\line(0,-1){10}}
\put(-20,15){\line(1,0){5}}
\put(0,15){\line(-1,0){5}}
\put(-13,14.4){$\dots$}
\put(-20,35){\circle*{1.5}}
\put(-25,37.5){\circle*{1.5}}
\put(-35,42.5){\circle*{1.5}}
\put(-22.5,36.25){\circle*{1}}
\put(-10,30){\line(-2,1){10}}
\put(-25,37.5){\line(-2,1){10}}
\put(-12,26){\scriptsize$v_{1}$}
\put(-5,22){\scriptsize$v_{2}$}
\put(-2,11.5){\scriptsize$v_{3}$}
\put(-25,11.5){\scriptsize$v_{g-1}$}
\put(-20,23){\scriptsize$v_{g}$}
\put(-32,42){\scriptsize$P_{l_{1}+1}$}
\put(-10,40){\circle*{1.5}}
\put(0,35){\circle*{1.5}}
\put(-7,38.5){\circle*{1}}
\put(-5,37.5){\circle*{1}}
\put(-3,36.5){\circle*{1}}
\put(-10,30){\line(0,1){10}}
\put(-10,30){\line(2,1){10}}
\put(-5,38.5){\scriptsize$a_{1}$}

\put(0,25){\line(2,-1){10}}
\put(0,25){\line(2,1){10}}
\put(10,30){\circle*{1.5}}
\put(10,20){\circle*{1.5}}
\put(10,27){\circle*{1}}
\put(10,25){\circle*{1}}
\put(10,23){\circle*{1}}
\put(11,25){\scriptsize$a_{2}$}

\put(-30,20){\circle*{1.5}}
\put(-30,30){\circle*{1.5}}
\put(-30,27){\circle*{1}}
\put(-30,25){\circle*{1}}
\put(-30,23){\circle*{1}}
\put(-20,25){\line(-2,-1){10}}
\put(-20,25){\line(-2,1){10}}
\put(-36,25){\scriptsize$a_{g}$}

\put(50,30){\circle*{1.5}}
\put(40,25){\circle*{1.5}}
\put(40,15){\circle*{1.5}}
\put(60,25){\circle*{1.5}}
\put(60,15){\circle*{1.5}}
\put(50,30){\line(-2,-1){10}}
\put(50,30){\line(2,-1){10}}
\put(40,25){\line(0,-1){10}}
\put(60,25){\line(0,-1){10}}
\put(40,15){\line(1,0){5}}
\put(60,15){\line(-1,0){5}}
\put(47,14.4){$\dots$}
\put(40,35){\circle*{1.5}}
\put(35,37.5){\circle*{1.5}}
\put(25,42.5){\circle*{1.5}}
\put(37.5,36.25){\circle*{1}}
\put(50,30){\line(-2,1){10}}
\put(35,37.5){\line(-2,1){10}}
\put(48,26){\scriptsize$v_{1}$}
\put(55,22){\scriptsize$v_{3}$}
\put(58,11.5){\scriptsize$v_{4}$}
\put(35,11.5){\scriptsize$v_{g-2}$}
\put(35,22){\scriptsize$v_{g-1}$}
\put(18,47){\scriptsize$P_{l_{1}+2}$}
\put(50,40){\circle*{1.5}}
\put(60,35){\circle*{1.5}}

\put(15,47.5){\circle*{1.5}}
\put(15,47.5){\line(2,-1){10}}
\put(12,44){\scriptsize$v_{2}$}

\put(53,38.5){\circle*{1}}
\put(55,37.5){\circle*{1}}
\put(57,36.5){\circle*{1}}
\put(50,30){\line(0,1){10}}
\put(50,30){\line(2,1){10}}
\put(55,38.5){\scriptsize$a_{1}+a_{2}+a_{g}$}

\put(60,30){\circle*{1.5}}
\put(50,30){\line(1,0){10}}
\put(61,29){\scriptsize$v_{g}$}

\put(50,5){\scriptsize$G'$}
\put(-13,5){\scriptsize$G$}
\put(13,3){\scriptsize$(c)$}
\put(-11,-5){\scriptsize$\rm{Fig.\, 4.}$ $G$ and $G'$ in Lemma \ref{lem3.5}}

\put(55,80){\circle*{1.5}}
\put(45,75){\circle*{1.5}}
\put(45,65){\circle*{1.5}}
\put(65,75){\circle*{1.5}}
\put(65,65){\circle*{1.5}}
\put(55,80){\line(-2,-1){10}}
\put(55,80){\line(2,-1){10}}
\put(45,75){\line(0,-1){10}}
\put(65,75){\line(0,-1){10}}
\put(45,65){\line(1,0){5}}
\put(65,65){\line(-1,0){5}}
\put(52,64.4){$\dots$}
\put(45,85){\circle*{1.5}}
\put(40,87.5){\circle*{1.5}}
\put(30,92.5){\circle*{1.5}}
\put(42.5,86.25){\circle*{1}}
\put(55,80){\line(-2,1){10}}
\put(40,87.5){\line(-2,1){10}}
\put(53,76){\scriptsize$v_{2}$}
\put(60,72){\scriptsize$v_{3}$}
\put(63,61.5){\scriptsize$v_{4}$}
\put(45,61.5){\scriptsize$v_{g}$}
\put(45,73){\scriptsize$v_{1}$}
\put(33,92){\scriptsize$P_{l_{1}+1}$}
\put(55,90){\circle*{1.5}}
\put(65,85){\circle*{1.5}}
\put(58,88.5){\circle*{1}}
\put(60,87.5){\circle*{1}}
\put(62,86.5){\circle*{1}}
\put(55,80){\line(0,1){10}}
\put(55,80){\line(2,1){10}}
\put(60,88.5){\scriptsize$a_{2}$}


\put(35,70){\circle*{1.5}}
\put(35,80){\circle*{1.5}}
\put(35,77){\circle*{1}}
\put(35,75){\circle*{1}}
\put(35,73){\circle*{1}}
\put(45,75){\line(-2,-1){10}}
\put(45,75){\line(-2,1){10}}
\put(29,75){\scriptsize$a_{1}$}

\put(35,65){\circle*{1.5}}
\put(45,65){\line(-1,0){10}}
\put(40,55){\circle*{1.5}}
\put(45,65){\line(-1,-2){5}}
\put(36.85,62){\circle*{1}}
\put(37.5,60){\circle*{1}}
\put(38.5,58){\circle*{1}}
\put(31.5,59.5){\scriptsize$a_{g}$}

\put(100,80){\circle*{1.5}}
\put(90,75){\circle*{1.5}}
\put(90,65){\circle*{1.5}}
\put(110,75){\circle*{1.5}}
\put(110,65){\circle*{1.5}}
\put(100,80){\line(-2,-1){10}}
\put(100,80){\line(2,-1){10}}
\put(90,75){\line(0,-1){10}}
\put(110,75){\line(0,-1){10}}
\put(90,65){\line(1,0){5}}
\put(110,65){\line(-1,0){5}}
\put(97,64.4){$\dots$}
\put(90,85){\circle*{1.5}}
\put(85,87.5){\circle*{1.5}}
\put(75,92.5){\circle*{1.5}}
\put(87.5,86.25){\circle*{1}}
\put(100,80){\line(-2,1){10}}
\put(85,87.5){\line(-2,1){10}}
\put(98,76){\scriptsize$v_{1}$}
\put(105,72){\scriptsize$v_{3}$}
\put(107,61.5){\scriptsize$v_{4}$}
\put(85,61.5){\scriptsize$v_{g-2}$}
\put(85,72){\scriptsize$v_{g-1}$}
\put(78,92){\scriptsize$P_{l_{1}+1}$}
\put(100,90){\circle*{1.5}}
\put(110,85){\circle*{1.5}}

\put(103,88.5){\circle*{1}}
\put(105,87.5){\circle*{1}}
\put(107,86.5){\circle*{1}}
\put(100,80){\line(0,1){10}}
\put(100,80){\line(2,1){10}}
\put(105,88.5){\scriptsize$a_{1}+a_{2}+a_{g}$}


\put(110,80){\circle*{1.5}}
\put(95,95){\circle*{1.5}}
\put(100,80){\line(1,0){10}}
\put(100,80){\line(-1,3){5}}
\put(111,79){\scriptsize$v_{g}$}
\put(95,95){\scriptsize$v_{2}$}

\end{picture} \end{center}


\begin{lemma}\label{lem3.5}

Let $g \geq 5$, $l_{1} \geq 1$, $a_{1}, a_{2}, a_{g}$ and $l_{2}$ be nonnegative integers. Let $C_{g}=v_{1}v_{2}  \cdots v_{g}v_{1}$ and $C'_{g-2}=v_{1}v_{3}v_{4}  \cdots v_{g-1}v_{1}$. Let $G=C_{g}(T_{1}, T_{2}, \cdots , T_{g})$ and $G'=C'_{g-2}(T'_{1}, T'_{3}, T'_{4}, \cdots , T'_{g-1})$, where $|E(T'_{1})|=|E(T_{1})|+|E(T_{2})|+|E(T_{g})|+2$, and
$T_{i}=T'_{i}$ for each $3 \leq i \leq g-1$.
\\
{\em(i)}  If $T_{1}=P_{l_{1}, l_{2}}^{a_{1}}$, $T_{2}=S_{a_{2}+1}$, $T_{g}=S_{1}$ and $T'_{1}=P_{l_{1}, l_{2}}^{a_{1}+a_{2}+2}$ ($G$ and $G'$ are depicted in Fig. 4(a)),
then $Sz_{e}(G') < Sz_{e}(G)$ and $Sz^{*}_{e}(G') < Sz^{*}_{e}(G)$;
\\
{\em(ii)} If $T_{1}=S_{a_{1}+1}$, $T_{2}=P_{l_{1}}^{a_{2}}$, $T_{g}=S_{a_{g}+1}$ and $T'_{1}=P_{l_{1}}^{a_{1}+a_{2}+a_{g}+2}$ ($G$ and $G'$ are depicted in Fig. 4(b)), then
$Sz_{e}(G') < Sz_{e}(G)$ and $Sz^{*}_{e}(G') < Sz^{*}_{e}(G)$;
\\
{\em(iii)} If $T_{1}=P_{l_{1}}^{a_{1}}$, $T_{2}=S_{a_{2}+1}$, $T_{g}=S_{a_{g}+1}$, $T'_{1}=P_{l_{1}+1}^{a_{1}+a_{2}+a_{g}+1}$ and $\sum_{j=3}^{g-1}|E(T_{j})|\geq l_{1}$
($G$ and $G'$ are depicted in Fig. 4(c)), then $Sz_{e}(G') < Sz_{e}(G)$ and $Sz^{*}_{e}(G') < Sz^{*}_{e}(G)$.

\end{lemma}

\begin{proof}  For convenience,
denote $|V(G)|=|V(G')|=n$, $d_{i,j}=d(v_i,v_j|C_g)$ for $i,j\in \{1,2,\cdots,g\}$,
 and $d'_{i,j}=d(v_i,v_j|C'_{g-2})$ for $i,j\in \{1,3,4,\cdots,g-1\}$. 
 For $1\leq i\leq g$ and $j\in \{1,3,4,\cdots,g-1\}$, let $|E(T_i)|=m_i$  and $|E(T'_j)|=m'_j$. Then \begin{eqnarray}m'_1=m_1+m_2+m_g+2.\label{E-5}\end{eqnarray}
By Lemma~\ref{lem2.3},  for (i) and (ii), we have
$$\sum\limits_{e\in E(G)\setminus E(C_{g})}m(e|G)=\sum\limits_{e \in E(G')\setminus E(C'_{g-2})}m(e|G')$$
and for (iii), one has that
$$\sum\limits_{e \in E(G)\setminus E(C_{g})}m(e|G)=\sum\limits_{e \in E(G')\setminus E(C'_{g-2})}m(e|G')-l_{1}(n-1-l_{1}).$$
Note that \emph{$d_{i,j}\geq d'_{i,j}$} for \emph{$3\leq i<j\leq g-1$} and $\delta(g)\leq 1$.
By Lemma \ref{lem2.8}, for (i) and (ii), we have
$$Sz_{e}(G)-Sz_{e}(G') \geq A+2B-C,$$ and
for (iii), we have
$$Sz_{e}(G)-Sz_{e}(G')\geq A+2B-C-D,$$
where
\begin{eqnarray*}
A&=&g ( \lceil \frac{g-2}{2} \rceil)^{2} + \lceil \frac{g-2}{2} \rceil g(n-g)-\delta(g)\lceil \frac{g-2}{2} \rceil(n-g)\\
&&-(g-2) ( \lceil \frac{g-4}{2} \rceil)^{2}
-\lceil \frac{g-4}{2} \rceil(n-g+2)(g-2)+\delta(g)\lceil \frac{g-4}{2} \rceil(n-g+2),\\
B&=&\sum\limits_{j=3}^{g-1}m_1m_jd_{1,j}-\sum\limits_{j=3}^{g-1}m_1'm_jd'_{1,j}+\sum\limits_{j=3}^{g-1}m_2m_jd_{2,j}
+\sum\limits_{j=3}^{g-1}m_gm_jd_{g,j}\\&&+m_1m_2d_{1,2}+m_1m_gd_{1,g}+m_2m_gd_{2,g},
\\
C&=&\sum\limits_{j=2}^{g}m_1m_j-\sum\limits_{j=3}^{g-1}m'_1m_j+\sum\limits_{j=3}^{g-1}m_2m_j
+\sum\limits_{j=3}^{g-1}m_gm_j+m_2m_g\\
&=&m_1m_2+m_1m_g+m_2m_g-2\sum\limits_{j=3}^{g-1}m_j,\\
\\
D&=&l_{1}(n-1-l_{1}).\end{eqnarray*}

The following two claims holds.

\noindent\textbf{Claim 1.} $A>\sum\limits_{j=3}^{g-1}m_j(-1+4d'_{1,j}) +(g-1)m_1.$

\noindent{\bf Proof of Claim 1:}
Let $\lceil \frac{g-2}{2} \rceil=x$. Then
$$A=g(n-g+1)+2(x-1)(n-g+2) -\delta(g)(n-g+2-2 x)+2 (x-1)^{2}.$$
As $n-g=\sum\limits_{j=1}^{g}m_j$,  $g\geq 2d'_{1,j}+2$, $x-1 \geq d_{1,j}'-1\geq 0$ and $\delta(g)\leq 1$,
\begin{eqnarray*}
g(n-g+1)&=&g\sum\limits_{j=3}^{g-1}m_j+g(m_1+m_2+m_g+1)
\geq\sum\limits_{j=3}^{g-1}m_j(2d'_{1,j}+2)+g(m_1+m_2+m_g+1),
\end{eqnarray*}
$$ 2(x-1)(n-g+2)\geq\sum\limits_{j=3}^{g-1}m_j(2d'_{1,j}-2), $$ and $$
\delta(g)(n-g+2-2 x)\leq n-g= \sum\limits_{j=1}^{g}m_j.$$
Hence
$$A\geq -\sum\limits_{j=3}^{g-1}m_j+ 4\sum\limits_{j=3}^{g-1}m_jd'_{1,j}+(g-1)(m_1+m_2+m_g)+g.$$
This completes the proof of Claim 1.

\noindent\textbf{Claim 2.} $B>\sum\limits_{j=3}^{g-1}m_j(m_1-2d'_{1,j})+m_1m_2+m_1m_g+2m_2m_g.$

\noindent{\bf Proof of Claim 2:} As $d_{1,2}=d_{1,g}=1$, $d_{2,g}=2$, and $d'_{1,j}=d_{1,j}-1$ for $j=3, 4, \cdots, g-1$,
by (\ref{E-5}),  we have
\begin{eqnarray*}
B&=&\sum\limits_{j=3}^{g-1}m_1m_jd_{1,j}
-\sum\limits_{j=3}^{g-1}m'_1m_j(d_{1,j}-1)+\sum\limits_{j=3}^{g-1}m_2m_jd_{2,j}
+\sum\limits_{j=3}^{g-1}m_gm_jd_{g,j}\\&&+m_1m_2d_{1,2}+m_1m_gd_{1,g}+m_2m_gd_{2,g}\\
&=&\sum\limits_{j=3}^{g-1}m_j[m_2(d_{2,j}-d_{1,j}+1)+m_g(d_{g,j}-d_{1,j}+1)+m_1-2d'_{1,j}]\\
&&+m_1m_2+m_1m_g+2m_2m_g.
\end{eqnarray*}
Let $\gamma=\sum\limits_{j=3}^{g-1}m_j[m_2(d_{2,j}-d_{1,j}+1)+m_g(d_{g,j}-d_{1,j}+1)].$  It is sufficient to prove $\gamma>0$.
By Lemma \ref{lem2.9}, if $g$ is an even number, we have
$
\gamma=\sum\limits_{j=3}^{\frac{g}{2}}2m_g m_j+\sum\limits_{j=\frac{g}{2}+2}^{g-1}2m_2m_j
\geq 0;$
if $g$ is an odd number,
\begin{eqnarray*}
\gamma&=&\sum\limits_{j=3}^{\frac{g-1}{2}}2m_g m_j+m_{\frac{g+1}{2}}m_{g}+m_{\frac{g+3}{2}}m_{2}
+\sum\limits_{j=\frac{g+5}{2}}^{g-1}2m_2 m_j \geq 0.
\end{eqnarray*}
This completes the proof of Claim 2.

By Claims 1 and 2, one has that
\begin{eqnarray*}
A+2B-C&>&(2m_1+1)\sum\limits_{j=3}^{g-1}m_j+m_1m_2+m_1m_g+g+(m_1+m_{2}+m_{g})(g-1)\\
&\geq&2m_1\sum\limits_{j=3}^{g-1}m_j+m_1m_2+m_1m_g+\sum\limits_{j=1}^{g}m_j+g+m_{1}(g-2)\\
&=&2m_1\sum\limits_{j=3}^{g-1}m_j+m_1m_2+m_1m_g+n+m_{1}(g-2).\end{eqnarray*}
Then, $Sz_{e}(G') < Sz_{e}(G)$.

Since $|E(T'_{1})|=|E(T_{1})|+|E(T_{2})|+|E(T_{g})|+2$, $|E(T'_{1})|^{2} > |E(T_{1})|^{2}+|E(T_{2})|^{2}+|E(T_{g})|^{2}$. From Lemma \ref{n1},
one has $Sz^{*}_{e}(G)-Sz^{*}_{e}(G') \geq Sz_{e}(G)-Sz_{e}(G')-n > 0$ holds immediately.

This completes the proof of the (i) and (ii) of Lemma~\ref{lem3.5}.

For (iii), it can be checked that
$$
D=l_{1}(n-1-l_{1})
=l_{1}(\sum\limits_{j=1}^{g}m_j+g-1-l_{1}) = l_{1}\sum\limits_{j=2}^{g}m_j+l_1m_1+l_1(g-1)-l^{2}_{1}.$$
By $l_1\leq m_1$ and $l_{1}\leq \sum_{j=3}^{g-1}m_j$, we have
\begin{eqnarray*}D&=&l_{1}\sum\limits_{j=3}^{g-1}m_j+l_1m_1+l_1m_2+l_1m_g+l_1(g-1)-l^{2}_{1}\\
&\leq&m_{1}\sum\limits_{j=3}^{g-1}m_j+m_1\sum\limits_{j=3}^{g-1}m_j+m_{1}m_{2}+m_{1}m_{g}+l_1(g-1)-l^{2}_{1}\\
&\leq&2m_1\sum\limits_{j=3}^{g-1}m_j+m_1m_2+m_1m_g+l_1g-2l_{1}\\
&\leq&2m_1\sum\limits_{j=3}^{g-1}m_j+m_1m_2+m_1m_g+l_{1}(g-2)\\
&\leq&2m_1\sum\limits_{j=3}^{g-1}m_j+m_1m_2+m_1m_g+m_{1}(g-2).\end{eqnarray*}

Hence for (iii), $Sz_{e}(G)-Sz_{e}(G')=A+2B-C-D>0$.

Since $|E(T'_{1})|=|E(T_{1})|+|E(T_{2})|+|E(T_{g})|+2$, $|E(T'_{1})|^{2} > |E(T_{1})|^{2}+|E(T_{2})|^{2}+|E(T_{g})|^{2}$. From Lemma \ref{n1},
one has $Sz^{*}_{e}(G)-Sz^{*}_{e}(G') \geq Sz_{e}(G)-Sz_{e}(G')-n > 0$ holds immediately.

This completes the proof of Lemma~\ref{lem3.5}(iii).
\end{proof}

\begin{theorem}\label{new2-1} Let $C_g=v_1v_2\cdots v_gv_1$ and $G_{0}=C_g(T_1,T_2,\cdots, T_g)$ be the graph with minimum revised edge Szeged index among the graphs in $\mathcal{U}_{n,d}$.
Let $P(G_{0})=u_0u_1u_2 \cdots u_{l_1-1}u_{l_1}(u_{l_1}=v_1)v_2\cdots v_l(v_l=w_{l_2})w_{l_2-1}\cdots w_2w_1w_0$ be a path of $G_{0}$ with length $d$, where $1 \leq l_{1} \le l_2$, $l_1+l+l_2=d+1$, $P'=u_0u_1\cdots u_{l_1-1}u_{l_1}$ be a path in $T_1$, and $P''=w_0w_1\cdots w_{l_2-1}w_{l_2}$ be a path in $T_l$.
Then each of the following holds:\\
{\em(i)} Let $|V(T_1)|=n_1$. If $l=1$, then $T_1=T_{n_1,d,\lfloor\frac{d}{2}\rfloor}=
P_{\lfloor\frac{d}{2}\rfloor,d-\lfloor\frac{d}{2}\rfloor}^{n_1-d-1}$;\\
{\em(ii)}If $l\not=1$, then $T_1=P_{l_1}^{n_1-l_1-1}$, and all vertices in $V(T_l)\setminus V(P'')$ are pendent vertices  adjacent to the same vertex $w_i$ ($1\leq i\leq l_2$);\\
{\em(iii)} All vertices in $V(G_{0})-(V(C_g)\cup V(P')\cup V(T_l))$ are pendent vertices and are adjacent to the same vertex $v_i$ with $1\leq i\leq g$ and $i\not=l$;\\
{\em(iv)} $g\leq 4$.
\end{theorem}

\begin{proof}
(i) By Lemma \ref{lem3.1}(ii), $T_1=T_l=T_{n_1,d,\lfloor\frac{d}{2}\rfloor}$.

(ii) If $l\not=1$, by Lemma \ref{lem3.1}(iii), $T_1$ and $T_l$ are the caterpillar graphs with backbones $P'$ and $P''$, respectively.
Suppose that $T_1=P'(0,a_1,a_2,\cdots,a_{l_1})$ and $T_l=P'(0,b_1,b_2,\cdots,b_{l_2})$.
We just need to prove the following two claims:

\noindent\textbf{Claim 1.} All vertices in  $V(T_1)\setminus V(P')$ are pendent vertices adjacent to $v_1$.

Suppose to the contrary that there is a vertex in $V(T_1)\setminus V(P')$ which is not adjacent to $v_1$, then take $k=\min\{i: a_i\not=0\}$. Denote by $X_k-1$ the edge size of the connected component of $G_{0}-u_k-u_{k+1}$ containing $u_{k-1}$, and denote $Y_k=n-X_k-a_k-a_{k+1}-1$. Since $l_1\leq l_2$, $Y_k+a_{k+1}>X_k$. By Lemma \ref{lem3.3}, we get a graph $G'\in\mathcal{U}_{n,d,g}$ such that $Sz_{e}^{*}(G')<Sz_{e}^{*}(G_{0})$, contrary to the assumption. Thus, claim 1 holds.

\noindent\textbf{Claim 2.} All vertices in $V(T_l)\setminus V(P'')$ are pendent vertices  adjacent to the same vertex $w_i$ ($1\leq i\leq l_2$).

Suppose to the contrary that there are at least two positive integers in $\{b_1,b_2,\cdots, b_{l_2}\}$. For $0\leq i\leq l_2$, denote by $X_i-1$ the edge size of the connected component of $G_{0}-w_i-w_{i+1}$ containing $w_{i-1}$, and denote $Y_i=n-X_i-b_i-b_{i+1}-1$. Suppose that $k=\min\{i: b_i\not=0\}$ and $j=\max\{i: b_{i+1}\not=0\}$.  It is routine to check that that $j\geq k$. By Lemma \ref{lem3.3}, we have $Y_k+b_{k+1}\leq X_k$  and $Y_k\leq X_k$. Since $j\geq k$, $X_j+b_j>X_k\geq Y_k\geq Y_j$. By Lemma \ref{lem3.3}, we get a graph $G''\in \mathcal{U}_{n,d,g}$ such that $Sz_{e}^{*}(G'')<Sz_{e}^{*}(G_{0})$, contrary to the assumption.
Thus, claim 2 holds.

(iii) By Lemma \ref{lem3.1}(i), for $2\leq i\leq g$ and $i\not=l$, $T_i$ is a star and $v_i$ is the center vertex of this star. By Lemma \ref{09}, (iii) holds immediately.

(iv) By contradiction, we assume that $g\geq 5$. If $l=1$, by Lemma \ref{lem3.5}(i), we get a graph $G'\in \mathcal{U}_{n,d,g-2}$ such that $Sz_{e}^{*}(G')<Sz_{e}^{*}(G_{0})$, contrary to the assumption.
If $l=2$ or $g$, by Lemma \ref{lem3.5}(ii), we get a graph $G'\in \mathcal{U}_{n,d,g-2}$ such that $Sz_{e}^{*}(G')<Sz_{e}^{*}(G_{0})$, contrary to the assumption. If $3\leq l\leq g-1$, by Lemma \ref{lem3.5}(iii), we get a graph $G'\in \mathcal{U}_{n,d,g-2}$ such that $Sz_{e}^{*}(G')<Sz_{e}^{*}(G_{0})$, contrary to  the assumption. Hence  $g\leq 4$.

 This completes the proof of  Theorem \ref{new2-1}.
\end{proof}

\section{The graphs with cycle length 3 and minimum revised edge Szeged index in $\mathcal{U}_{n,d}$}

In this section, the graphs in $\mathcal{U}_{n,d}$ with a cycle of length 3 and minimum revised edge Szeged index are identified. Firstly, we introduced some lemmas.

\begin{center}   \setlength{\unitlength}{0.7mm}
\begin{picture}(30,40)

\put(-25,20){\circle*{1.5}}
\put(-25,20){\line(1,2){5}}
\put(-25,20){\line(-1,2){5}}
\put(-30,15){\circle*{1.5}}
\put(-27,15){\circle*{1}}
\put(-25,15){\circle*{1}}
\put(-23,15){\circle*{1}}
\put(-20,15){\circle*{1.5}}
\put(-30,30){\circle*{1.5}}
\put(-25,20){\line(1,-1){5.2}}
\put(-25,20){\line(-1,-1){5.2}}
\put(-49,38){\circle*{1.5}}
\put(-41,38){\circle*{1.5}}
\put(-46,38){\circle*{1}}
\put(-44,38){\circle*{1}}
\put(-45,30){\line(-1,2){4}}
\put(-45,30){\line(1,2){4}}
\put(-40,30){\circle*{1.5}}
\put(-42.5,30){\circle*{1}}
\put(-45,30){\circle*{1.5}}
\put(-55,30){\circle*{1.5}}
\put(-60,30){\circle*{1.5}}
\put(-70,30){\circle*{1.5}}
\put(-57.5,30){\circle*{1.5}}
\put(-45,30){\line(-1,0){10}}
\put(-60,30){\line(-1,0){10}}
\put(-20,30){\circle*{1.5}}
\put(-30,30){\line(-1,0){10}}
\put(-30,30){\line(1,0){10}}
\put(-24,38){\circle*{1.5}}
\put(-16,38){\circle*{1.5}}
\put(-21,38){\circle*{1}}
\put(-19,38){\circle*{1}}
\put(-20,30){\line(-1,2){4}}
\put(-20,30){\line(1,2){4}}
\put(-20,30){\line(1,0){10}}
\put(-5,30){\line(1,0){10}}
\put(-10,30){\circle*{1.5}}
\put(-5,30){\circle*{1.5}}
\put(5,30){\circle*{1.5}}
\put(-7.5,30){\circle*{1}}
\put(-32,32){\scriptsize$v_{3}$}
\put(-18,31){\scriptsize$v_{1}$}

\put(-35,27){\scriptsize$w_{l_{2}}$}
\put(-20,27){\scriptsize$u_{l_{1}}$}

\put(-24,20){\scriptsize$v_{2}$}
\put(-26,10){\scriptsize$b$}

\put(-46,27){\scriptsize$w_{i}$}
\put(-63,27){\scriptsize$w_{1}$}
\put(-73,27){\scriptsize$w_{0}$}
\put(-73,33){\scriptsize$P_{l_{2}+1}$}
\put(-21,40){\scriptsize$a$}
\put(-46,40){\scriptsize$c$}
\put(-5,27){\scriptsize$u_{1}$}
\put(5,27){\scriptsize$u_{0}$}
\put(-3,33){\scriptsize$P_{l_{1}+1}$}

\put(-45,5){\scriptsize$G$}
\put(45,5){\scriptsize$G'$}
\put(-20,0){\scriptsize Fig. 5. Graphs $G$ and $G'$ in Lemma \ref{L3-Z14}}

\put(65,20){\circle*{1.5}}
\put(65,20){\line(1,2){5}}
\put(65,20){\line(-1,2){5}}
\put(60,15){\circle*{1.5}}
\put(63,15){\circle*{1}}
\put(65,15){\circle*{1}}
\put(67,15){\circle*{1}}
\put(70,15){\circle*{1.5}}
\put(60,30){\circle*{1.5}}
\put(65,20){\line(1,-1){5.2}}
\put(65,20){\line(-1,-1){5.2}}
\put(41,38){\circle*{1.5}}
\put(49,38){\circle*{1.5}}
\put(44,38){\circle*{1}}
\put(46,38){\circle*{1}}
\put(45,30){\line(-1,2){4}}
\put(45,30){\line(1,2){4}}
\put(50,30){\circle*{1.5}}
\put(47.5,30){\circle*{1}}
\put(45,30){\circle*{1.5}}
\put(35,30){\circle*{1.5}}
\put(30,30){\circle*{1.5}}

\put(22,30){\circle*{1.5}}

\put(22,30){\line(1,0){8}}

\put(105,30){\circle*{1.5}}
\put(32.5,30){\circle*{1.5}}
\put(45,30){\line(-1,0){10}}
\put(95,30){\line(1,0){10}}
\put(70,30){\circle*{1.5}}
\put(60,30){\line(-1,0){10}}
\put(60,30){\line(1,0){10}}
\put(66,38){\circle*{1.5}}
\put(74,38){\circle*{1.5}}
\put(69,38){\circle*{1}}
\put(71,38){\circle*{1}}
\put(70,30){\line(-1,2){4}}
\put(70,30){\line(1,2){4}}
\put(70,30){\line(1,0){10}}
\put(85,30){\line(1,0){10}}
\put(80,30){\circle*{1.5}}
\put(85,30){\circle*{1.5}}
\put(95,30){\circle*{1.5}}
\put(82.5,30){\circle*{1}}
\put(58,32){\scriptsize$v_{3}$}
\put(72,31){\scriptsize$v_{1}$}

\put(55,27){\scriptsize$w_{l_{2}-1}$}
\put(70,27){\scriptsize$u_{l_{1}}$}

\put(66,20){\scriptsize$v_{2}$}
\put(64,10){\scriptsize$b$}

\put(44,27){\scriptsize$w_{i}$}

\put(27,27){\scriptsize$w_{1}$}
\put(19,27){\scriptsize$w_{0}$}

\put(105,27){\scriptsize$w_{l_{2}}$}

\put(21,33){\scriptsize$P_{l_{2}}$}
\put(69,40){\scriptsize$a$}
\put(44,40){\scriptsize$c$}
\put(85,27){\scriptsize$u_{1}$}
\put(95,27){\scriptsize$u_{0}$}
\put(97,33){\scriptsize$P_{l_{1}+2}$}

\end{picture} \end{center}
\vspace{-0.5cm}

\begin{lemma}\label{L3-Z14}
Let $a,b,c,i,l_1,l_2\geq 0$ be integers, and $0< i< l_2$. Let $C_3=v_1v_2v_3v_1$ and $G=C_3(T_1,T_2, T_3)$, where
\emph{$T_1$} is the tree obtained from $P'=u_0u_1u_2 \cdots u_{l_1-1}u_{l_1}(=v_1)$ by attaching $a$ pendent vertices at $u_{l_1}(=v_1)$, $T_2=S_{b+1}$, and $T_3$ is the tree obtained from $P''=w_0w_1\cdots w_{l_2}(=v_3)$ by attaching $c$ pendent vertices at $w_{i}$. Let
$G'=G-\{w_{l_2}w_{l_2-1},w_{l_2}v_1,w_{l_2}v_2\}+
\{w_{l_2-1}v_1,w_{l_2-1}v_2,w_{l_2}u_0\}$ $(G$ and $G'$ are depicted in Fig. $5)$. If $l_2\geq l_1+1$, then $Sz_{e}^*(G')\leq Sz_{e}^*(G),$ where the equality holds if and only if $l_2= l_1+1$ and $a=c=0$.
\end{lemma}

\begin{proof}
Let $G'=C_3(T'_1,T'_2,T'_3)$ and $|V(G)|=|V(G')|=n$. Then \begin{eqnarray}T_2=T_2', |E(T_3)|-|E(T'_3)|=1 \mbox{ and } |E(T'_1)|-|E(T_1)|=1.\label{E-14}\end{eqnarray} By Corollary \ref{02}, we have
\begin{eqnarray*}Sz_{e}^*(G)-Sz_{e}^*(G')&=&Sz(G)-Sz(G')+\sum_{i=1}^3 D(v_{i}|T_{i})-\sum_{i=1}^r D(v_{i}|T'_{i}) \nonumber\\
&&+\frac{1}{4}(|E(T_1')|^2+|E(T_3')|^2-|E(T_1)|^2-|E(T_3)|^2).\label{E-11}\end{eqnarray*}
By (\ref{E-14}), one has that
\begin{eqnarray*}
|E(T_1')|^2+|E(T_3')|^2-|E(T_1)|^2-|E(T_3)|^2=2(|V(T_1)|-|V(T_3)|+1)=2(l_1+1+a-l_2-c).\label{E-12}
\end{eqnarray*}

It can be checked that
$$\sum_{i=1}^3 D(v_{i}|T_{i})-\sum_{i=1}^3 D(v_{i}|T'_{i})=l_{2}+c-l_{1}-1.$$

Let $E= E(C_3)\cup \{w_{l_2-1}w_{l_2}\}$ and $E'=E(C_3)\cup \{u_{l_1}u_{l_1-1}\}$. By the definition of Szeged index and (\ref{E-14}), we have
\begin{eqnarray*}Sz(G)-Sz(G')&=&\sum_{e=xy\in E}n_{x}(e|G)n_{y}(e|G)-\sum_{e=xy \in E'}n_{x}(e|G')n_{y}(e|G') \nonumber\\
&=&(|V(T_3)|-1)(n-|V(T_3)|+1)+|V(T_2)||V(T_3)|+|V(T_2)||V(T_1)|\\
&&+|V(T_1)||V(T_3)|-(l_1+1)(n-l_1-1)-|V(T'_2)||V(T'_3)|\\
&&-|V(T'_2)||V(T'_1)|-|V(T'_1)||V(T'_3)|\\
&=&(a+b+1)(c-l_1+l_2)-(b+1).\label{E-13}\end{eqnarray*}
Hence
\begin{eqnarray*}
Sz_{e}^*(G)-Sz_{e}^*(G')
=(b+\frac{3}{2})(c+l_2-l_1-1)+\frac{a}{2}(1+2c+2l_2-2l_1).
\end{eqnarray*}
Thus, $Sz_{e}^*(G) \geq Sz_{e}^*(G')$ with the equality holds if and only if $l_2= l_1+1$ and $a=c=0$.

This completes the proof of Lemma \ref{L3-Z14}.
\end{proof}

\begin{corollary}\label{L3-Z16}
Let $a,l_1,l_2\geq 0$ be integers, and $l_2 \geq l_{1}+3$. Then,
$$Sz_{e}^*(C_{3}(P_{l_{2}+1}, P_{l_{1}}^{a}, P_{1})) > Sz_{e}^*(C_{3}(P_{l_{2}}, P_{l_{1}+1}^{a}, P_{1})).$$
\end{corollary}

\begin{lemma}\label{L3-Z15}
Let $a,l_1,l_2\geq 0$ be integers, and $l_2 \geq l_{1}+3$. Then,
$$Sz_{e}^*(C_{3}(P_{l_{1}+1}, P_{l_{2}}^{a}, P_{1})) > Sz_{e}^*(C_{3}(P_{l_{1}+2}, P_{l_{2}-1}^{a}, P_{1})).$$
\end{lemma}

\begin{proof}
Let $G=C_{3}(P_{l_{1}+1}, P_{l_{2}}^{a}, P_{1})$ and $G'=C_{3}(P_{l_{1}+2}, P_{l_{2}-1}^{a}, P_{1})$.
By Lemma \ref{n1}, one has
\begin{eqnarray*}Sz_{e}^*(G)-Sz_{e}^*(G')&=&Sz_{e}(G)-Sz_{e}(G')+\frac{1}{4}[\sum_{i=1}^3 |E(T_{i}')|^{2}-\sum_{i=1}^3 |E(T_{i})|^{2}]\\
&=&(l_{2}-1)(l_{1}+a+3)+(l_{1}+1)(l_{2}+a+1)+(l_{1}+l_{2}+a+2)\\
&&-l_{1}(l_{2}+a+2)+(l_{1}+2)(l_{2}+a)+(l_{1}+l_{2}+a+2)\\
&&+ \frac{1}{4}[(l_{1}+1)^{2}+(a+l_{2}-1)^{2}-l_{1}^{2}-(a+l_{2})^{2}]\\
&=&a(l_{2}-l_{1}-\frac{5}{2})+\frac{3}{2}(l_{2}-l_{1}-1)> 0.\end{eqnarray*}

This completes the proof of Lemma \ref{L3-Z15}.
\end{proof}


Let $G^{3}_{11}=C_{3}(P_{l_{1}}^{a}, P_{l_{2}+1}, S_{1})$, $G^{3}_{12}=C_{3}(P_{l_{1}+1}, P_{l_{2}+1}, S_{a+1})$, $G^{3}_{21}=C_{3}(P_{l_{1}, l_{2}+1}^{a-1}, S_{1}, S_{1})$ and $G^{3}_{22}=C_{3}(P_{l_{1}, l_{2}+1}^{0}, S_{a}, S_{1})$ be the graphs in $\mathcal{U}_{n, d}$ (which are depicted in Fig. 6), where $a \geq 1$,
$l_{1}+l_{2}+1=d$ and $l_{1}+l_{2}+a+3=n$. We have the following result.

\begin{center}   \setlength{\unitlength}{0.7mm}
\begin{picture}(100,50)

\put(-35,30){\circle*{1.5}}
\put(-40,20){\circle*{1.5}}
\put(-30,20){\circle*{1.5}}
\put(-35,30){\line(1,0){10}}
\put(-25,30){\circle*{1.5}}
\put(-30,40){\circle*{1.5}}
\put(-35,30){\line(1,2){5}}
\put(-40,20){\line(1,2){5}}
\put(-30,20){\line(-1,2){5}}
\put(-40,20){\line(1,0){10}}
\put(-20,20){\circle*{1.5}}
\put(-15,20){\circle*{1.5}}
\put(-5,20){\circle*{1.5}}
\put(-30,20){\line(1,0){10}}
\put(-15,20){\line(1,0){10}}
\put(-17.5,20){\circle*{1}}
\put(-45,30){\circle*{1.5}}
\put(-50,30){\circle*{1.5}}
\put(-60,30){\circle*{1.5}}
\put(-35,30){\line(-1,0){10}}
\put(-50,30){\line(-1,0){10}}
\put(-47.5,30){\circle*{1}}
\put(-28.5,37){\circle*{1}}
\put(-27.5,35){\circle*{1}}
\put(-26.5,33){\circle*{1}}
\put(-26,35){\scriptsize$a$}
\put(-45,21.5){\scriptsize$v_{3}$}
\put(-31,21.5){\scriptsize$v_{2}$}
\put(-40,31.5){\scriptsize$v_{1}$}
\put(-38,10){\scriptsize$G^{3}_{11}$}
\put(-62,33){\scriptsize$P_{l_{1}+1}$}
\put(-15,23){\scriptsize$P_{l_{2}+1}$}

\put(20,30){\circle*{1.5}}
\put(15,20){\circle*{1.5}}
\put(25,20){\circle*{1.5}}
\put(15,20){\line(1,2){5}}
\put(25,20){\line(-1,2){5}}
\put(15,20){\line(1,0){10}}
\put(35,20){\circle*{1.5}}
\put(40,20){\circle*{1.5}}
\put(50,20){\circle*{1.5}}
\put(25,20){\line(1,0){10}}
\put(40,20){\line(1,0){10}}
\put(37.5,20){\circle*{1}}
\put(10,30){\circle*{1.5}}
\put(5,30){\circle*{1.5}}
\put(-5,30){\circle*{1.5}}
\put(20,30){\line(-1,0){10}}
\put(5,30){\line(-1,0){10}}
\put(7.5,30){\circle*{1}}
\put(10,21.5){\scriptsize$v_{3}$}
\put(24,21.5){\scriptsize$v_{2}$}
\put(15,31.5){\scriptsize$v_{1}$}
\put(17,10){\scriptsize$G^{3}_{12}$}
\put(-7,33){\scriptsize$P_{l_{1}+1}$}
\put(40,23){\scriptsize$P_{l_{2}+1}$}

\put(5,20){\circle*{1.5}}
\put(10,10){\circle*{1.5}}
\put(15,20){\line(-1,0){10}}
\put(15,20){\line(-1,-2){5}}
\put(6.5,17){\circle*{1}}
\put(7.5,15){\circle*{1}}
\put(8.5,13){\circle*{1}}
\put(3.5,12){\scriptsize$a$}

\put(75,30){\circle*{1.5}}
\put(70,20){\circle*{1.5}}
\put(80,20){\circle*{1.5}}
\put(75,30){\line(1,0){10}}
\put(85,30){\circle*{1.5}}
\put(80,40){\circle*{1.5}}
\put(75,30){\line(1,2){5}}
\put(70,20){\line(1,2){5}}
\put(80,20){\line(-1,2){5}}
\put(70,20){\line(1,0){10}}
\put(65,30){\circle*{1.5}}
\put(60,30){\circle*{1.5}}
\put(50,30){\circle*{1.5}}
\put(75,30){\line(-1,0){10}}
\put(60,30){\line(-1,0){10}}
\put(62.5,30){\circle*{1}}
\put(81.5,37){\circle*{1}}
\put(82.5,35){\circle*{1}}
\put(83.5,33){\circle*{1}}
\put(84,35){\scriptsize$a-1$}
\put(65,21.5){\scriptsize$v_{3}$}
\put(74,21.5){\scriptsize$v_{2}$}
\put(70,31.5){\scriptsize$v_{1}$}
\put(72,10){\scriptsize$G^{3}_{21}$}
\put(48,33){\scriptsize$P_{l_{1}+1}$}

\put(65,35){\circle*{1.5}}
\put(60,37.5){\circle*{1.5}}
\put(50,42.5){\circle*{1.5}}
\put(75,30){\line(-2,1){10}}
\put(62.5,36.25){\circle*{1}}
\put(60,37.5){\line(-2,1){10}}
\put(50,42.5){\scriptsize$P_{l_{2}+2}$}

\put(130,30){\circle*{1.5}}
\put(125,20){\circle*{1.5}}
\put(135,20){\circle*{1.5}}
\put(135,20){\line(1,0){10}}
\put(145,20){\circle*{1.5}}
\put(140,30){\circle*{1.5}}
\put(135,20){\line(1,2){5}}
\put(125,20){\line(1,2){5}}
\put(135,20){\line(-1,2){5}}
\put(125,20){\line(1,0){10}}
\put(120,30){\circle*{1.5}}
\put(115,30){\circle*{1.5}}
\put(105,30){\circle*{1.5}}
\put(130,30){\line(-1,0){10}}
\put(115,30){\line(-1,0){10}}
\put(117.5,30){\circle*{1}}
\put(141.5,27){\circle*{1}}
\put(142.5,25){\circle*{1}}
\put(143.5,23){\circle*{1}}
\put(144,25){\scriptsize$a-1$}
\put(120,21.5){\scriptsize$v_{3}$}
\put(129,21.5){\scriptsize$v_{2}$}
\put(125,31.5){\scriptsize$v_{1}$}
\put(127,10){\scriptsize$G^{3}_{22}$}
\put(103,33){\scriptsize$P_{l_{1}+1}$}

\put(120,35){\circle*{1.5}}
\put(115,37.5){\circle*{1.5}}
\put(105,42.5){\circle*{1.5}}
\put(130,30){\line(-2,1){10}}
\put(117.5,36.25){\circle*{1}}
\put(115,37.5){\line(-2,1){10}}
\put(105,42.5){\scriptsize$P_{l_{2}+2}$}

\put(15,0){\scriptsize Fig. 6. Graphs $G^{3}_{11}$ $G^{3}_{12}$, $G^{3}_{21}$ and $G^{3}_{22}$}

\end{picture} \end{center}

\begin{lemma} \label{new3-1}
Each of the following holds:\\
{\em(i)} $Sz^{*}_{e}(G^{3}_{11}) \leq  Sz^{*}_{e}(G^{3}_{12})$ with the equality holds if and only if $al_{1}=0$;\\
{\em(ii)} $Sz^{*}_{e}(G^{3}_{21}) \leq  Sz^{*}_{e}(G^{3}_{22})$ with the equality holds if and only if $a=1$;\\
{\em(iii)} If $l_{2}=0$, then $G^{3}_{11} \cong G^{3}_{21}$; if $a, l_1, l_2 \geq 1$, $a +l_{1}+l_{2}  > 10$ and $l_{1} \leq l_{2} < l_{1}+3$, then $Sz^{*}_{e}(G^{3}_{11})  \geq  Sz^{*}_{e}(G^{3}_{21})$.
\end{lemma}

\begin{proof} (i) By Lemma \ref{lem2.3} and the definition of the edge Szeged index, we have
\begin{eqnarray*}Sz_{e}(G^{3}_{11}) - Sz_{e}(G^{3}_{12})
&=&\sum \limits_{e\in  E(C_{3})  } m(e|G^{3}_{11}) -\sum\limits_{e\in E(C_{3}) }m(e|G^{3}_{12})\\
&=&(l_{1}+l_{2}+a+2)+(l_{1}+a+1)(l_{2}+1)\\
&&-(a+1)(l_{1}+l_{2}+2)-(l_{1}+1)(l_{2}+1)\\
&=&-al_{1}.
\end{eqnarray*}

By Lemma \ref{n1}, one has
$$Sz^{*}_{e}(G^{3}_{11})- Sz^{*}_{e}(G^{3}_{12})=Sz_{e}(G^{3}_{11}) - Sz_{e}(G^{3}_{12})-\frac{1}{4}[(a+l_{1})^{2}+l^{2}_{2}-a^{2}-l^{2}_{1}-l^{2}_{2}]=-\frac{3}{2}al_{1}.$$
Thus, (i) holds.

(ii)  From Lemma \ref{lem2.5}, one has $Sz_{e}(G^{3}_{21}) \leq Sz_{e}(G^{3}_{22})$ with the equality holds if and only if $a=1$.

By Lemma \ref{n1}, one has
$$Sz^{*}_{e}(G^{3}_{22})- Sz^{*}_{e}(G^{3}_{21})=Sz_{e}(G^{3}_{22}) - Sz_{e}(G^{3}_{21})+\frac{1}{4}[(a+l_{1}+l_{2})^{2}-(a-1)^{2}-(l_{1}+l_{2}+1)^{2}] \geq 0$$
with the equality holds if and only if $a=1$.

(iii) It is easy to see that $G^{3}_{11} \cong G^{3}_{21}$ if $l_{2}=0$. Note that $|E(G^{3}_{11})|=|E(G^{3}_{21})|=n$. By the definition of edge Szeged index,
one has that
\begin{eqnarray*}Sz_{e}(G^{3}_{11})&=&(l_{1}+l_{2}+a+2)+(a+l_{1}+1)(l_{2}+1)+\sum\limits_{i=1}^{l_{1}-1}i(n-1-i)+\sum\limits_{j=1}^{l_{2}-1}j(n-1-j),
\\
Sz_{e}(G^{3}_{21})&=&1+2(a+l_{1}+l_{2}+1)+\sum\limits_{i=1}^{l_{1}-1}i(n-1-i)+\sum\limits_{j=1}^{l_{2}}j(n-1-j).\end{eqnarray*}
Then $Sz_{e}(G^{3}_{11})-Sz_{e}(G^{3}_{21})=-2l_{2}$.

By Lemma \ref{n1}, one has
$$Sz^{*}_{e}(G^{3}_{11})- Sz^{*}_{e}(G^{3}_{21})=Sz_{e}(G^{3}_{11}) - Sz_{e}(G^{3}_{21})+\frac{1}{4}[(a+l_{1}+l_{2})^{2}-l^{2}_{2}-(l_{1}+a)^{2}]=\frac{1}{2} l_{2}(a+l_{1}-4).$$
Since $a +l_{1}+l_{2}  > 10$ and $l_{1} \leq l_{2} < l_{1}+3$, $a+l_{1} > 4$. Thus, $Sz^{*}_{e}(G^{3}_{11})  >  Sz^{*}_{e}(G^{3}_{21})$.

Hence, Lemma \ref{new3-1} holds.
\end{proof}

\begin{theorem}\label{main-L-1}
Let $G_{0}$ be the nonbipartite graph with minimum revised edge Szeged index among the graphs in $\mathcal{U}_{n,d}$ with order $n \geq 15$.\\
{\em(i)} If $d=n-2$, then
 $G_{0}=C_3( P_{\lceil \frac{d-1}{2} \rceil+1},P_{\lfloor \frac{d-1}{2} \rfloor+1},S_{1})$;\\
{\em(ii)} If $3\leq d\leq n-3$, then $G_{0}=C_3(P_{\lfloor \frac{d}{2} \rfloor ,d- \lfloor \frac{d}{2} \rfloor}^{n-d-3},S_1,S_1)$.
\end{theorem}
\begin{proof}
From Theorem \ref{new2-1}, the unique cycle of $G_{0}$ is a 3-cycle $v_{1}v_{2}v_{3}v_{1}$ and  one can suppose that $G_{0}=C_3(T_1,T_2, T_3)$.
Let $P(G_{0})=u_0u_1u_2 \cdots u_{l_1-1}u_{l_1}(=v_1)v_2\cdots v_l(=w_{l_2})w_{l_2-1}\cdots w_2w_1w_0$ be a path in $G_{0}$ with length $d$,
where $1 \leq l_{1} \le l_2$, $l_1+l+l_2=d+1$, $P'=u_0u_1\cdots u_{l_1-1}u_{l_1}$ be a path in $T_1$ and $P''=w_0,w_1,\cdots,w_{l_2-1}w_{l_2}$ be a path in $T_l$.

(i) If $d=n-2$, then $G_{0} = C_3( P_{l_1+1},P_{l_2+1},S_{1})$ for some integers $l_1,l_2$ with $l_1+l_2=n-3$.
By Lemma \ref{L3-Z14}, $G_{0}=C_3( P_{\lceil \frac{d-1}{2} \rceil+1},P_{\lfloor \frac{d-1}{2} \rfloor+1},S_{1})$.

(ii) Let $3\leq d\leq n-3$. We prove the following claim firstly.

\noindent\textbf{Claim:} If $l\not=1$ and $V(T_l)\setminus V(P'')\not=\emptyset$, then all vertices in $V(T_l)\setminus V(P'')$ are pendent vertices and adjacent to the same vertex $v_l(=w_{l_2})$.

If $l_2=l_1$, by Theorem \ref{new2-1}, the claim holds immediately. If $l_2\geq l_1+1$, $V(T_l)\setminus V(P'')\not=\emptyset$ and all vertices in $V(T_l)\setminus V(P'')$ are pendent vertices  adjacent to the same vertex $w_i$ ($1\leq i< l_2$), then by Lemma \ref{L3-Z14}, there exists a graph $G'$ such that $Sz_{e}^*(G')<Sz_{e}^*(G_{0})$, contrary to the assumption. Thus, the claim holds.

Hence if $l\not=1$, by Lemmas \ref{09} and \ref{new3-1}(i), all vertices in $ V(G^*)\setminus(V(C_g)\cup V(P')\cup V(P''))$ are pendent vertices and adjacent to the same vertex $v_j\in V(C_3)$ $(j\in \{1,l\})$.
Thus, for $3\leq d\leq n-3$ and $l\not=1$, by Lemma \ref{L3-Z15} and Corollary \ref{L3-Z16}, one has $l_{1} \leq l_{2} < l_{1}+3$.
From Lemma \ref{new3-1}(iii),
there exists a graph $G''$ such that $Sz_{e}^*(G_{0})  >  Sz_{e}^*(G'')$, which contradicts to the minimum of $Sz_{e}^*(G_{0})$.
Thus, one can suppose $l=1$ in the following.

Since $l=1$, by Theorem \ref{new2-1} and Lemma \ref{new3-1}(ii),  all vertices in $ V(G^*)\setminus(V(C_g)\cup V(P')\cup V(P''))$ are pendent vertices adjacent to the same vertex $v_1$.
Then, $G_{0}=C_3(P_{\lfloor \frac{d}{2} \rfloor ,d- \lfloor \frac{d}{2} \rfloor}^{n-d-3},S_1,S_1)$.

This completes the proof of Theorem \ref{main-L-1}.
\end{proof}

\section{The graphs with cycle length 4 and minimum revised edge Szeged index in $\mathcal{U}_{n,d}$}
In this section, the graphs in $\mathcal{U}_{n,d}$ with a cycle of length 4 and minimum revised edge Szeged index are identified. Firstly, we introduced some lemmas.

\begin{lemma}\label{L4-1} Let $C_4=v_1v_2v_3v_4v_1$ and \emph{$G^0=C_4(T_1,T_2, T_3,T_4)$}. Then each of the following holds:\\
{\em(i)} Let $T'$ be the tree obtained from $T_1$, $T_2$, $T_3$ and $T_4$ by identifying their root vertices to a new vertex which is the root vertex of $T'$ and $G^1=C_{4}(T',S_1,S_1,S_1)$. If $|V(T_1)|>1$ and there is at least one number in $\{|V(T_2)|,|V(T_3)|,|V(T_4)|\}$ is more than $1$, then $Sz_{e}^*(G^0)>Sz_{e}^*(G^{1})$.\\
{\em(ii)} Let $T''$ be the tree obtained from $T_1$, $T_3$ and $T_4$ by identifying their root vertices to a new vertex which is the root vertex of $T''$, and let $G^2=C_{4}(T'',T_2,S_1,S_1)$. If $|V(T_1)|,|V(T_2)|>1$ and there is at least one of $|V(T_3)|$ and $|V(T_4)|$ is more than $1$, then $Sz_{e}^*(G^0)>Sz_{e}^*(G^{2})$.\\
{\em(iii)} Let $T'''$ be the tree obtained from $T_1$, $T_2$ and $T_4$ by identifying their root vertices to a new vertex which is the root vertex of $T'''$, and $G^{3}=C_{4}(T''',S_1,T_3,S_1)$. Let $|V(T_4)|=1$ and $|V(T_2)|>1$. If $|V(T_1)|\leq |V(T_3)|$ , then
 \emph{$Sz_{e}^*(G^0)\leq Sz_{e}^*(G^{3})$}, where the equality holds if and only if $|V(T_1)|= |V(T_3)|.$
 \end{lemma}

\begin{proof} Let $m_{i}=|E(T_{i})|$ for $1 \leq i \leq 4$. By the definition of revised edge Szeged index, one has
$$\sum\limits_{e=xy \not\in E(C_{4})}m^{*}(e|G^{0})=\sum\limits_{e=xy \not\in E(C_{4})}m^{*}(e|G^{1})=\sum\limits_{e=xy \not\in E(C_{4})}m^{*}(e|G^{2})=\sum\limits_{e=xy \not\in E(C_{4})}m^{*}(e|G^{3}).$$

(i) By the definition of revised edge Szeged index, one has
\begin{eqnarray*}
Sz_{e}^*(G^0)-Sz_{e}^*(G^{1})
&=&\sum\limits_{e=xy \in E(C_{4})}m^{*}(e|G^{0})- \sum\limits_{e=xy \in E(C_{4})}m^{*}(e|G^{1})\\
&=&2(m_{1}+m_{4}+2)(m_{2}+m_{3}+2)+2(m_{1}+m_{2}+2)(m_{3}+m_{4}+2)\\
&&-8(m_{1}+m_{2}+m_{3}+m_{4}+2)\\
&=&2[(m_{1}+m_{2})(m_{3}+m_{4})+(m_{1}+m_{4})(m_{2}+m_{3})].
\end{eqnarray*}
By the condition $|V(T_1)|>1$ and there is at least one number in $\{|V(T_2)|,|V(T_3)|,|V(T_4)|\}$ is more than $1$, one has that $Sz_{e}^*(G^0)>Sz_{e}^*(G^{1})$.

(ii) By the definition of revised edge Szeged index, one has
\begin{eqnarray*}
Sz_{e}^*(G^0)-Sz_{e}^*(G^{2})
&=&\sum\limits_{e=xy \in E(C_{4})}m^{*}(e|G^{0})- \sum\limits_{e=xy \in E(C_{4})}m^{*}(e|G^{2})\\
&=&2(m_{1}+m_{4}+2)(m_{2}+m_{3}+2)+2(m_{1}+m_{2}+2)(m_{3}+m_{4}+2)\\
&&-4(m_{1}+m_{2}+m_{3}+m_{4}+2)+2(m_{1}+m_{3}+m_{4}+2)(m_{2}+2)\\
&=&4m_{1}m_{3}+2m_{4}(m_{1}+m_{2}+m_{3}).
\end{eqnarray*}
Since $|V(T_1)|,|V(T_2)|>1$ and there is at least one of $|V(T_3)|$ and $|V(T_4)|$ is more than $1$,  $Sz_{e}^*(G^0)>Sz_{e}^*(G^{2})$.

(iii) By the definition of revised edge Szeged index, one has
\begin{eqnarray*}
Sz_{e}^*(G^0)-Sz_{e}^*(G^{3})
&=&\sum\limits_{e=xy \in E(C_{4})}m^{*}(e|G^{0})- \sum\limits_{e=xy \in E(C_{4})}m^{*}(e|G^{3})\\
&=&2(m_{1}+m_{4}+2)(m_{2}+m_{3}+2)+2(m_{1}+m_{2}+2)(m_{3}+m_{4}+2)\\
&&-4(m_{1}+m_{2}+m_{4}+2)(m_{3}+2)\\
&=&2m_{2}(m_{1}-m_{3}).
\end{eqnarray*}
Then, $Sz_{e}^*(G^0) \leq Sz_{e}^*(G^{3})$  with the equality holds if and only if $m_1= m_3.$
\end{proof}

\begin{definition}
Let $l_1$, $l_2$ and $b$ be integers with $l_2\geq l_1$. Let $P'=u_0u_1u_2 \cdots u_{l_1-1}u_{l_1}$ be a path with root vertex $u_{l_1}$, and $T^*$ be the tree obtained from a path $P''=w_0w_1\cdots w_{l_2}$ by attaching $b$ pendent vertices at $w_{i}$ $(1\leq i\leq l_2)$ with root vertex $w_{l_2}$.\\
 {\em (i)} Let $C_4=v_1v_2v_3v_4v_1$, $G_1(l_1,l_2)=C_4(P_{l_1,l_2}^0,S_1, S_1,S_1)$, $G_2(l_1,l_2)=C_4(P',T^* ,S_1,S_1)$ and $G_3(l_1,l_2)=C_4(P',S_1,T^*,S_1)$. \\
 {\em (ii)} For integers $a$ and $k$ with $1\leq k\leq 4$ and $a$, let  $G^{4}_{1k}(l_1,l_2,a)$ be the graph obtained from $G_1(l_1,l_2)$ by attaching $a$ pendent vertices at $v_k$, and $G^{4}_{2k}(l_1,l_2,a,b,)$ (resp., $G^{4}_{3k}(l_1,l_2,a,b,i)$) be the the graph obtained from $G_2(l_1,l_2)$ (resp., $G_3(l_1,l_2)$) by attaching $a$ pendent vertices at $v_k$.
 \end{definition}

If $G$ is the bipartite graph in $\mathcal{U}_{n,d}$ with minimum revised edge Szeged index, by Theorem \ref{new2-1} and Lemma \ref{L4-1}(iii),
$G$ is isomorphic to one of $G^{4}_{1k}(l_1,l_2,a)$ ($k\in \{1,2,3\}$), $G^{4}_{2k}(l_1,l_2,a,b,i)$ ($k\in \{1,3,4\}$) and $G^{4}_{3k}(l_1,l_2,a,b,i)$ ($k\in \{1,2\}$)
for some integers $l_1,l_2,i,a$ and $b$. If $b=0$, $G^{4}_{2k}(l_1,l_2,a,b,i)\cong G^{4}_{2k}(l_1,l_2,a,b,j)$ and $G^{4}_{3k}(l_1,l_2,a,b,i)\cong G^{4}_{3k}(l_1,l_2,a,b,j)$ for any $j\not=i$ and $(1\leq j\leq l_2)$.
So in this case, sometimes we write $G^{4}_{2k}(l_1,l_2,a,0,i)$ as $G^{4}_{2k}(l_1,l_2,a)$, and write $G^{4}_{3k}(l_1,l_2,a,0,i)$) as $G^{4}_{3k}(l_1,l_2,a)$. By Lemma \ref{L4-1}, we have the following results.

\begin{corollary}\label{C4-1}Let $l_1,l_2 ,a,b,i\geq 0$ be integers, and $a\geq 1$.\\
{\em(i)} If $l_1+l_2\geq 1,$ then $Sz_{e}^*(G^{4}_{11}(l_1,l_2,a))<Sz_{e}^*(G^{4}_{1k}(l_1,l_2,a))$ for $k\in \{2,3\}$;\\
{\em(ii)} If $l_1\geq 1$ and $l_2\geq 1,$ then $Sz_{e}^*(G^{4}_{21}(l_1,l_2,a,b,i))<Sz_{e}^*(G^{4}_{2k}(l_1,l_2,a,b,i))$ for $k\in \{3,4\}$;\\
{\em(iii)} If $l_2\geq l_1$, then $Sz_{e}^*(G^{4}_{31}(l_1,l_2,a,b,i))\geq Sz_{e}^*(G^{4}_{32}(l_1,l_2,a,b,i))$, where the equality holds if and only if  $l_1=l_2$ and $b=0$.
\end{corollary}

\begin{lemma}\label{L4-3}
Let $l_2 \geq l_1,a,b,i\geq 0$ be integers, and $0< i< l_2$. If $l_2 \geq l_1+2$ or $a+b > 0$, then $Sz_{e}^*(G^{4}_{21}(l_1,l_2,a,b,i))>Sz_{e}^*(G^{4}_{21}(l_1+1,l_2-1,a,b,i)) $ .
\end{lemma}

\begin{proof}
Let $G=G^{4}_{21}(l_1,l_2,a,b,i)$ and $G'=G^{4}_{21}(l_1+1,l_2-1,a,b,i)$. Let $|V(G)|=|V(G')|=n$, $E=\{w_{l_2-1}w_{l_2},v_1v_2,v_3v_4\}$ and $E'=\{u_{l_1}u_{l_1+1},v_1v_2,v_3v_4\}$.
By Lemma \ref{n1} and $n=a+b+l_1+l_2+4$, we have
\begin{eqnarray*}Sz_{e}^{*}(G)-Sz_{e}^{*}(G')&=&\sum_{e=xy\in E }m(e|G) -\sum_{e=xy\in E'}m(e|G')\nonumber\\
&=&2(a+l_1+1)(b+l_2+1)+(b+l_2-1)(l_1+a+4)
\\
&&-2(a+l_1+2)(b+l_2)-l_1(a+4+l_2-1+b)\\
&=&a(b+l_2+1-l_1)+2b+2(l_2-l_1-1).\end{eqnarray*}

This completes the proof of Lemma \ref{L4-3}.
\end{proof}

\begin{lemma}\label{L4-4-1}
Let $a,b,i,l_1,l_2\geq 0$ be integers such that $1\leq i\leq l_2$, $l_{2} \geq l_1$ and $a\geq 1$.
If $a+b+l_{1}+l_{2}+4 \geq 14$, then
$$Sz_{e}^*(G^{4}_{21}(l_{1}, l_{2}, a, b, i)) > Sz_{e}^*(G^{4}_{11}(l_{1}+1, l_{2}, a-1, b, i)).$$
\end{lemma}
\begin{proof}

Let $G=G^{4}_{21}(l_1,l_2,a,b,i)$ and $G'=G^{4}_{11}(l_{1}+1, l_{2}, a-1, b, i)$.
Let $|V(G)|=|V(G')|=n$, $E=E(C_4)$ and $E'=\{u_{l_1}u_{l_1+1}\}\cup E(C_4)$.
By Lemma \ref{n1} and $n=a+b+l_1+l_2+4$, we have
\begin{eqnarray*}Sz_{e}^{*}(G)-Sz_{e}^{*}(G')&=&\sum_{e=xy\in E }m(e|G) -\sum_{e=xy\in E'}m(e|G')\nonumber\\
&=&2(l_{1}+l_{2}+a+b+1)+2(l_1+a+1)(b+l_{2}+1)\\
&&-4(l_{1}+l_{2}+a+b+1)-l_{1}(b+l_{2}+a+3)\\
&=&l_{1}(a+b+l_{2}-3)+2ab+a(l_{2}-l_{1}).\end{eqnarray*}
As $l_{1}+l_{2}+a+b+4=n \geq 14$ and $l_2 \geq l_1$, then $a+b+l_{2} \geq 5$ and $Sz_{e}^{*}(G)-Sz_{e}^{*}(G') > 0$.

Thus, Lemma \ref{L4-4-1} holds.
\end{proof}

\begin{lemma}\label{L4-4-2}
Let $b,l_1,l_2\geq 0$ be integers such that $l_{2} \geq l_1 \geq 1$ and $b\geq 1$.
If $b+l_{1}+l_{2}+4 \geq 14$, then
$$Sz_{e}^*(G^{4}_{21}(l_{1}, l_{2}, 0, b, l_{2})) > Sz_{e}^*(G^{4}_{11}(l_{1}+1, l_{2}, b-1)).$$
\end{lemma}
\begin{proof}

Let $G=G^{4}_{21}(l_{1}, l_{2}, 0, b, l_{2})$ and $G'=G^{4}_{11}(l_{1}+1, l_{2}, b-1)$. Let $E=E(C_4)$ and $E'=\{u_{l_1}u_{l_1+1}\}\cup E(C_4)$.
By Lemma \ref{n1}, we have
\begin{eqnarray*}Sz_{e}^{*}(G)-Sz_{e}^{*}(G')&=&\sum_{e=xy\in E }m(e|G) -\sum_{e=xy\in E'}m(e|G')\nonumber\\
&=&2(l_{1}+1)(l_{2}+b+1)+2(l_1+b+l_{2}+1)\\
&&-4(l_{1}+l_{2}+b+1)-l_{1}(b+l_{2}+3)\\
&=&l_{1}(b+l_{2}-3).\end{eqnarray*}
As $l_{1}+l_{2}+b+4\geq 14$ and $l_2  \geq l_{1}$, then $b+l_{2} \geq 5$ and $Sz_{e}^{*}(G)-Sz_{e}^{*}(G') > 0$.

Thus, Lemma \ref{L4-4-2} holds.
\end{proof}

\begin{lemma}\label{L4-3-2}
Let $l_{2} \geq l_{1} \geq 1,a,b,i\geq 0$ be integers and $0< i< l_2$. Then $Sz_{e}^*(G^{4}_{32}(l_1,l_2,a,b,i)) \leq Sz_{e}^*(G^{4}_{32}(l_1-1,l_2+1,a,b,i))$, where the equality holds if and only if  $a=0$.
\end{lemma}

\begin{proof}
Let $G=G^{4}_{32}(l_1,l_2,a,b,i)$ and $G'=G^{4}_{32}(l_1-1,l_2+1,a,b,i) $. Let $E=\{u_{l_1-1}u_{l_1}\}\cup E(C_4)$, and $E'=\{w_{l_2}w_{l_2+1}\}\cup E(C_4)$.
By Lemma \ref{n1}, we have
\begin{eqnarray*}Sz_{e}^{*}(G)-Sz_{e}^{*}(G')&=&\sum_{e=xy\in E }m(e|G) -\sum_{e=xy\in E'}m(e|G')\nonumber\\
&=&(l_1-1)(a+l_2+b+4)+2(a+l_1+1)(l_2+b+1)\\
&&+2(l_{1}+1)(a+b+l_{2}+1)-(l_1-1+4+a)(l_2+b)\\
&&-2l_1(a+l_2+b+2)-2(l_{1}+a)(l_{2}+b+2)\\
&=&a(l_{1}-b-l_{2}-1).\label{E-13}\end{eqnarray*}
Thus, $Sz_{e}^*(G) \leq Sz_{e}^*(G')$ with equality holds if and only if $a=0$. This completes the proof of Lemma \ref{L4-3-2}.
\end{proof}

\begin{lemma}\label{L4-4}
 Let $a,b,i,l_2 \geq  l_1\geq 1$ be integers such that $1\leq i\leq l_2$ and $a\geq 1$.
If $a+b+l_{1}+l_{2}+4 \geq 14$, then $Sz_{e}^*(G^{4}_{32}(l_1,l_2,a,b,i)) > Sz_{e}^*(G^{4}_{21}(l_1+1,l_2,a-1,b,i))$.

\end{lemma}

\begin{proof}
(i) Let $G=G^{4}_{32}(l_1,l_2,a,b,i)$ and $G'=G^{4}_{21}(l_1+1,l_2,i,a-1,b)$. Let $|V(G)|=|V(G')|=n$, $E= E(C_4)$, and $E'=\{u_{l_1}u_{l_1+1}\}\cup E(C_4)$.
By Lemma \ref{n1} and $n=a+b+l_1+l_2+4$, we have
\begin{eqnarray*}Sz_{e}^{*}(G)-Sz_{e}^{*}(G')&=&\sum_{e=xy\in E }m(e|G) -\sum_{e=xy\in E'}m(e|G')\nonumber\\
&=&2(l_1+1)(l_2+b+a+1)+2(l_1+a+1)(l_2+b+1)\\
&&-l_{1}(l_{2}+a+b+3)-2(l_1+a+1)(l_2+b+1)-2(l_1+l_{2}+a+b+1)\\
&=&l_{1}(l_{2}+a+b-3).\end{eqnarray*}
As $l_{1}+l_{2}+a+b+4=n \geq 14$ and $l_2 \geq l_1$, then $a+b+l_{2} \geq 5$ and $Sz_{e}^{*}(G)-Sz_{e}^{*}(G')> 0$.

Thus,  Lemma \ref{L4-4} holds.
\end{proof}

\begin{lemma}\label{d=n-3} Let $n \geq 14$ and $d=n-3$. Then
$$Sz_{e}^*(G^{4}_{32}(0,d-2,0,1, \lceil\frac{d+1}{2}  \rceil )) > Sz_{e}^*(G^{4}_{21}(\lfloor \frac{d-1}{2} \rfloor , \lceil \frac{d-1}{2} \rceil,0)).$$
\end{lemma}
\begin{proof}
Let $G=G^{4}_{32}(0,d-2,0,1, \lceil \frac{d+1}{2}  \rceil )$ and $G'=G^{4}_{21}(\lfloor \frac{d-1}{2} \rfloor , \lceil \frac{d-1}{2} \rceil,0)$.
We divided into two cases according to the parity of $d$.

{\bf  Case 1.} $d=2k$.

It implies that $k \geq 5$ and $G=G^{4}_{32}(0,2k-2,0,1, k+1)$ and $G'=G^{4}_{21}(k-1, k,0)$.
By Lemma \ref{n1}, one has
\begin{eqnarray*}Sz_{e}^{*}(G)-Sz_{e}^{*}(G')&=&\sum_{e=xy\in E(G) }m(e|G) -\sum_{e=xy\in E(G')}m(e|G')\nonumber\\
&=&4(n-3)+\sum_{i=1}^{k}i(n-1-i)+\sum_{i=4}^{k}i(n-1-i)\\
&&-2(n-3)-2k(k+1)-\sum_{i=1}^{k-1}i(n-1-i)-\sum_{i=1}^{k-2}i(n-1-i)\\
&=&k^{2}-4k-1>0.\end{eqnarray*}

{\bf  Case 2.} $d=2k+1.$

It implies that $k \geq 5$ and $G=G^{4}_{32}(0,2k-1,0,1, k+1)$ and $G'=G^{4}_{21}(k, k,0)$.
By Lemma \ref{n1}, one has that
\begin{eqnarray*}Sz_{e}^{*}(G)-Sz_{e}^{*}(G')&=&\sum_{e=xy\in E(G) }m(e|G) -\sum_{e=xy\in E(G')}m(e|G')\nonumber\\
&=&4(n-3)+\sum_{i=1}^{k}i(n-1-i)+\sum_{i=4}^{k+1}i(n-1-i)\\
&&-2(n-3)-2(k+1)(k+1)-\sum_{i=1}^{k-1}i(n-1-i)-\sum_{i=1}^{k-1}i(n-1-i)\\
&=&k^{2}-3k-2>0.\end{eqnarray*}

This completes the proof.
\end{proof}

\begin{lemma}\label{d=n-3-2} Let $n > 14$ and $d=n-3$. Then
$$Sz_{e}^*(G^{4}_{32}( \lfloor \frac{d-2}{2}  \rfloor, \lceil  \frac{d-2}{2}  \rceil ,1)) > Sz_{e}^*(G^{4}_{21}(\lfloor \frac{d-1}{2}  \rfloor, \lceil  \frac{d-1}{2}  \rceil, 0 )).$$
\end{lemma}
\begin{proof}
Let $G=G^{4}_{32}( \lfloor \frac{d-2}{2}  \rfloor, \lceil  \frac{d-2}{2}  \rceil ,1)$ and $G'=G^{4}_{21}(\lfloor \frac{d-1}{2}  \rfloor, \lceil  \frac{d-1}{2}  \rceil , 0)$. We divided into two cases according to the parity of $d$.

{\bf  Case 1.} $d=2k$.

Then $k \geq 5$ and $G=G^{4}_{32}(k-1,k-1,1)$ and $G'=G^{4}_{21}(k-1, k, 0) $.
Let $E= E(C_4)$ and $E'=\{w_{k-1}w_{k}\}\cup E(C_4)$.
By Lemma \ref{n1}, one has
\begin{eqnarray*}Sz_{e}^{*}(G)-Sz_{e}^{*}(G')&=&\sum_{e=xy\in E }m(e|G) -\sum_{e=xy\in E'}m(e|G')\nonumber\\
&=&4k(k+1)-2k(k+1)-4k-(k-1)(k+3)\\
&=&(k-1)(k-3)> 0.\end{eqnarray*}

{\bf  Case 2.} $d=2k+1$.

Then $k \geq 5$ and $G=G^{4}_{32}(k-1,k,1)$ and $G'=G^{4}_{21}(k, k, 0)$.
Let $E= E(C_4)$ and $E'=\{u_{k-1}u_{k}\}\cup E(C_4)$.
By Lemma \ref{n1}, one has
\begin{eqnarray*}Sz_{e}^{*}(G)-Sz_{e}^{*}(G')&=&\sum_{e=xy\in E }m(e|G) -\sum_{e=xy\in E'}m(e|G')\nonumber\\
&=&2k(k+2)+2(k+1)^{2}-2(k+1)^{2}-2(2k+1)-(k-1)(k+4)\\
&=&(k-1)(k-2) > 0.\end{eqnarray*}

This completes the proof.
\end{proof}


\begin{theorem}\label{main-L-2}
Let $G_{0}$ be the bipartite graph with minimum revised edge Szeged index among the graphs in $\mathcal{U}_{n,d}$. \\
{\em(i)} If $d=n-2$, then
 $G_{0}=G^{4}_{32}(r_{1},r_{2}, 0)$, where $r_{1}$ and $r_{2}$ are any nonnegative integers such that $r_{1}+r_{2}=n-4$;\\
{\em(ii)} If $d=n-3$, then $G_{0}=G^{4}_{21}(\lfloor \frac{d-1}{2}  \rfloor,\lceil \frac{d-1}{2} \rceil, 0)$;\\
{\em(iii)} If $4\leq d\leq n-4$, then $G_{0} \in \{  G^{4}_{32}(l_{1},l_{2},0,n-d-2, i)  ,G^{4}_{11}(\lfloor \frac{d}{2} \rfloor ,d- \lfloor \frac{d}{2} \rfloor,n-d-4)\}$, where $l_{1}, l_{2}$ and $i$ are some nonnegative integers such that $l_{1}+l_{2}=d-2$ and $0 < i < l_{2}$;\\
{\em(iv)} If $d=3$, then $G_{0}=C_4(S_{n-3},S_1,S_1,S_1)$.
\end{theorem}

\begin{proof} By Theorem \ref{new2-1}, the unique cycle of $G_{0}$ is a 4-cycle $v_{1}v_{2}v_{3}v_{4}v_{1}$ and  one can suppose that $G_{0}=C_4(T_1,T_2, T_3, T_{4})$.
Let $P(G_{0})=u_0u_1u_2 \cdots u_{l_1-1}u_{l_1}(=v_1)v_2\cdots v_l(=w_{l_2})w_{l_2-1}\cdots w_2w_1w_0$ be a path in $G_{0}$ with length $d$, where $l_{1} \le l_2$, $l_1+l+l_2=d+1$, $P'=u_0u_1\cdots u_{l_1-1}u_{l_1}$ is a path in $T_1$, and $P''=w_0w_1\cdots w_{l_2-1}w_{l_2}$ is a path in $T_l$.

(i) If $d=n-2$, from Lemma \ref{L4-3-2}, (i) holds immediately.

(ii) If $d = n-3$, then $G_{0}=G^{4}_{21}(k_{1},k_{2}, 0) $,  where $k_{1}$ and $k_{2}$ are some nonnegative integers such that $k_{1}+k_{2}=d-1$,  or $G_{0}=G^{4}_{32}(k_{3},k_{4}, a, b,i)$, where $k_{3}, k_{4}, a, b, i$ are some nonnegative integers such that $k_{3}+k_4=d-2$ and $ a+b=1 $, or $G_{0}=G^{4}_{31}(k_{5},k_{6}, 1)$ where $k_{5}$ and $k_{6}$ are some nonnegative integers such that $k_{5}+k_{6}=d-2$.

By Corollary \ref{C4-1} and Lemmas \ref{L4-3-2} and \ref{d=n-3-2}, one has $Sz_{e}^{*}(G^{4}_{31}(k_{5},k_{6}, 1)) \geq Sz_{e}^{*}(G^{4}_{32}(k_{5},k_{6}, 1)) \geq Sz_{e}^*(G^{4}_{32}( \lfloor \frac{d-2}{2}  \rfloor, \lceil  \frac{d-2}{2}  \rceil ,1)) \rceil > Sz_{e}^*(G^{4}_{21}(\lfloor \frac{d-1}{2}  \rfloor, \lceil  \frac{d-1}{2}  \rceil, 0 ))$. Then, $G_{0} \neq  G^{4}_{31}(k_{5},k_{6}, 1)$ and $G_{0} \neq  G^{4}_{32}(k_{3},k_{4}, 1)$.

By Lemmas \ref{L4-3-2}, \ref{lem3.3} and \ref{d=n-3}, one has that  $Sz_{e}^{*}(G^{4}_{32}(k_{3},k_{4}, 0, 1, i)) =Sz_{e}^{*}(G^{4}_{32}(0,d-2, 0, 1, i))
\geq Sz_{e}^*(G^{4}_{32}(0,d-2,0,1, \lceil\frac{d+1}{2}  \rceil ))  > Sz_{e}^*(G^{4}_{21}(\lfloor \frac{d-1}{2} \rfloor , \lceil \frac{d-1}{2} \rceil,0))$.
Then, $G_{0} \neq  G^{4}_{32}(k_{3},k_{4}, 0, 1, i)$.
Thus, $G_{0}=G^{4}_{21}(k_{1},k_{2}, 0) $ where $k_{1}$ and $k_{2}$ are some nonnegative integers such that $k_{1}+k_{2}=d-1$.
By Lemma \ref{L4-3}, one has $k_{1}= \lfloor \frac{d-1}{2} \rfloor$ and $k_{2}= \lceil \frac{d-1}{2} \rceil$.
Then, if $d=n-3$, one has $G_{0}=G^{4}_{21}(\lfloor \frac{d-1}{2}  \rfloor,\lceil \frac{d-1}{2} \rceil, 0)$.

(iii)  $4 \leq d \leq n-4$.

By Theorem \ref{new2-1} and  Corollary \ref{C4-1}, we have
$G_{0} \in \{ G^{4}_{11}(l_{1}, l_{2}, a), G^{4}_{21}(l_{1}, l_{2}, a, b, i),  G^{4}_{32}(l_{1}, l_{2}, a,\\ b, i) \}$.

If $l=2$, by Corollary \ref{C4-1}(ii), $G_{0}$ is isomorphic to $G^{4}_{21}(l_1,l_2,a,b,i)$, where $l_1+l_2+a+b+4=n$ and $l_2\geq l_1\geq 1$.
If $a \geq 1$, by Lemma \ref{L4-4-1}, $Sz_{e}^*(G^{4}_{21}(l_1,l_2,a,b,i)) > Sz_{e}^*(G^{4}_{11}(l_1+1,l_2,a-1,b,i))$ which is contrary to the
assumption. Then $a=0$ and $b \geq 1$.

If $l_2=l_1$, then by  Theorem \ref{new2-1}, we have $i=l_2$, i.e., all pendent vertices in $V(T_l)\setminus V(P'')$ are adjacent to $w_{l_2}$.
If $l_2\geq l_1+1$ and $i< l_2$, then by Lemma \ref{L4-3}, we have $Sz_{e}^*(G^{4}_{21}(l_1,l_2,a,b,i))>Sz_{e}^*(G^{4}_{21}(l_1+1,l_2-1,a,b,i))$,
contrary to assumption. Hence in this case, we also have $i=l_2$. By Lemma \ref{L4-4-2},
one has $Sz_{e}^*(G_{21}(l_{1}, l_{2}, 0, b, l_{2})) > Sz_{e}^*(G_{11}(l_{1}+1, l_{2}, b-1))$, contrary to the
assumption. Then, $l=1$ or $l=3$.

If $l=1$. By Theorem \ref{new2-1}, we have $G_{0}= G^{4}_{11}(\lfloor \frac{d}{2} \rfloor ,d- \lfloor \frac{d}{2} \rfloor,n-d-4)$.

If $l=3$. By Corollary  \ref{C4-1}, we have $G_{0}=G^{4}_{32}(l_{1}, l_{2}, a, b, i)$. If $a \geq 1$, by Lemma \ref{L4-4}, $Sz_{e}^*(G^{4}_{32}(l_1,l_2,a,b,i)) > Sz_{e}^*(G^{4}_{21}(l_1+1,l_2,a-1,b,i))$,
contrary to the
assumption. Then $a=0$. By Lemma \ref{L4-3-2}, one has $G_{0}=G^{4}_{32}(l_{1},l_{2},0,n-d-2, i)$ $(l_{1}+l_{2}=d-2)$.

Thus, if $4\leq d\leq n-4$,  $G_{0}$ is a graph in $\{  G^{4}_{32}(l_{1},l_{2},0,n-d-2, i)$ $(l_{1}+l_{2}=d-2) ,G^{4}_{11}(\lfloor \frac{d}{2} \rfloor ,d- \lfloor \frac{d}{2} \rfloor,n-d-4)\}$ with minimum revised edge Szeged index.

(iv) If $d=3$, by Lemmas \ref{L4-1}(i) and \ref{lem3.1}(i), we have $G_{0}=C_{4}(S_{n-3}, S_{1}, S_{1}, S_{1})$.

This completes the proof of Theorem \ref{main-L-2}.
\end{proof}

\section{The proof of Theorem \ref{Th1}}

In this section, the proof of Theorem \ref{Th1} is presented. By Theorems \ref{main-L-1} and \ref{main-L-2}, it is sufficient to compare the minimum revised edge Szeged indices of the graphs in $\mathcal{U}_{n,d}$ with cycle length 3 and the minimum revised edge Szeged indices of the graphs in $\mathcal{U}_{n,d}$ with cycle length 4. Firstly, some lemmas are given.

\begin{lemma}\label{L71} If $n>15$ and $d=n-2$, then $Sz_{e}^*(C_3( P_{\lceil \frac{d-1}{2} \rceil+1},P_{\lfloor \frac{d-1}{2} \rfloor+1},S_{1}))  < Sz_{e}^*(G^{4}_{32}(0,n-4,0))$.
\end{lemma}

\begin{proof} Let $G= C_3( P_{\lceil \frac{d-1}{2} \rceil+1},P_{\lfloor \frac{d-1}{2} \rfloor+1},S_{1})$ and $G'= G^{4}_{32}(0,n-4,0)$.
We divided into two cases according to the parity of $d$.

{\bf Case 1.} $d=2k$.

 It is routine to check that $C_3( P_{\lceil \frac{d-1}{2} \rceil+1},P_{\lfloor \frac{d-1}{2} \rfloor+1},S_{1})=C_3( P_{k+1},P_{k},S_{1})$. By Lemma \ref{L4-3-2}, one has that $Sz_{e}^*(G^{4}_{32}(0,n-4,0))=Sz_{e}^*(G^{4}_{32}(k-1,k-1,0))$.
By the definition of revised edge Szeged index, we have
\begin{eqnarray*}Sz_{e}^{*}(G)-Sz_{e}^{*}(G')&=&\sum_{e=xy\in E(G) }m^{*}(e|G) -\sum_{e=xy\in E(G')}m^{*}(e|G')\nonumber\\
&=&\sum_{i=0}^{k-2}(i+\frac{1}{2})(n-1-i+\frac{1}{2})+\sum_{i=0}^{k-1}(i+\frac{1}{2})(n-1-i+\frac{1}{2})\\
&&+(1+\frac{k+1}{2})(k+\frac{k+1}{2})+(k+\frac{1}{2})(k+1+\frac{1}{2})+(1+\frac{k}{2})(k+1+\frac{k}{2})\\
&&-2\sum_{i=0}^{k-2}(i+\frac{1}{2})(n-1-i+\frac{1}{2})-4(k+\frac{2}{2})(k+\frac{2}{2})\\
&=&\frac{1}{4}(2k-2k^{2}-11) < 0.\end{eqnarray*}

{\bf Case 2.} $d=2k+1$.

Obviously, $C_3( P_{\lceil \frac{d-1}{2} \rceil+1},P_{\lfloor \frac{d-1}{2} \rfloor+1},S_{1})=C_3( P_{k+1},P_{k+1},S_{1})$. By Lemma \ref{L4-3-2}, $Sz_{e}^*(G^{4}_{32}(0,n-4,0))=Sz_{e}^*(G^{4}_{32}(k-1,k,0))$.
By the definition of revised edge Szeged index, we have
\begin{eqnarray*}Sz_{e}^{*}(G)-Sz_{e}^{*}(G')&=&\sum_{e=xy\in E(G) }m^{*}(e|G) -\sum_{e=xy\in E(G')}m^{*}(e|G')\nonumber\\
&=&2\sum_{i=0}^{k-1}(i+\frac{1}{2})(n-1-i+\frac{1}{2})+(k+1+\frac{1}{2})(k+1+\frac{1}{2})\\
&&+(1+\frac{k+1}{2})(k+1+\frac{k+1}{2})+(1+\frac{k+1}{2})(k+1+\frac{k+1}{2})\\
&&-\sum_{i=0}^{k-2}(i+\frac{1}{2})(n-1-i+\frac{1}{2})-\sum_{i=0}^{k-1}(i+\frac{1}{2})(n-1-i+\frac{1}{2})\\
&&-4(k+\frac{2}{2})(k+1+\frac{2}{2})\\
&=&\frac{1}{2}(-k^{2}-6) < 0.\end{eqnarray*}

Thus, the proof is completed.
\end{proof}

\begin{lemma}\label{L72} If $n>14$ and $d=n-3$, then
$$Sz_{e}^*(C_3(P_{\lfloor \frac{d}{2} \rfloor ,d- \lfloor \frac{d}{2} \rfloor}^{n-d-3},S_1,S_1))<Sz_{e}^*(G^{4}_{21}(\lfloor \frac{d-1}{2}  \rfloor,\lceil \frac{d-1}{2} \rceil, 0)).$$
\end{lemma}

\begin{proof} Let $G= G^{4}_{21}(\lfloor \frac{d-1}{2}  \rfloor,\lceil \frac{d-1}{2} \rceil, 0)$ and $G'=C_3(P_{\lfloor \frac{d}{2} \rfloor ,d- \lfloor \frac{d}{2} \rfloor}^{n-d-3},S_1,S_1)$. We divided into two cases according to the parity of $d$.

{\bf Case 1.} $d=2k$.

Obviously, $G=G^{4}_{21}(k,k-1, 0)$ and $G'=C_3(P_{k,k},S_1,S_1)$. By the definition of revised edge Szeged index, we have
\begin{eqnarray*}Sz_{e}^{*}(G)-Sz_{e}^{*}(G')&=&\sum_{e=xy\in E(G) }m^{*}(e|G) -\sum_{e=xy\in E(G')}m^{*}(e|G')\nonumber\\
&=&\sum_{i=0}^{k-1}(i+\frac{1}{2})(n-1-i+\frac{1}{2})+\sum_{i=0}^{k-2}(i+\frac{1}{2})(n-1-i+\frac{1}{2})\\
&&+2(k+1+\frac{2}{2})(k+\frac{2}{2})+2(1+\frac{2}{2})(2k+\frac{2}{2})\\
&&-2\sum_{i=0}^{k-1}(i+\frac{1}{2})(n-1-i+\frac{1}{2})-(1+\frac{2k+1}{2})(1+\frac{2k+1}{2})\\
&&-2(1+\frac{1}{2})(2k+1+\frac{1}{2})\\
&=&\frac{1}{4}(8k+13) > 0.\end{eqnarray*}

{\bf Case 2.} $d=2k+1$.

Obviously, $G=G^{4}_{21}(k,k, 0)$ and $G'=C_3(P_{k,k+1},S_1,S_1)$. By the definition of revised edge Szeged index, we have
\begin{eqnarray*}Sz_{e}^{*}(G)-Sz_{e}^{*}(G')&=&\sum_{e=xy\in E(G) }m^{*}(e|G) -\sum_{e=xy\in E(G')}m^{*}(e|G')\nonumber\\
&=&\sum_{i=0}^{k-1}(i+\frac{1}{2})(n-1-i+\frac{1}{2})+\sum_{i=0}^{k-1}(i+\frac{1}{2})(n-1-i+\frac{1}{2})\\
&&+2(k+1+\frac{2}{2})(k+1+\frac{2}{2})+2(1+\frac{2}{2})(2k+1+\frac{2}{2})\\
&&-\sum_{i=0}^{k-1}(i+\frac{1}{2})(n-1-i+\frac{1}{2})-\sum_{i=0}^{k}(i+\frac{1}{2})(n-1-i+\frac{1}{2})\\
&&-(1+\frac{2k+2}{2})(1+\frac{2k+2}{2})-2(1+\frac{1}{2})(2k+2+\frac{1}{2})\\
&=&\frac{1}{4}(8k+11) > 0.\end{eqnarray*}

This completes the proof.
\end{proof}

\begin{lemma}\label{L73}
If $n>15$ and $4 \leq  d  \leq n-4$, then
$$Sz_{e}^*(C_3(P_{\lfloor \frac{d}{2} \rfloor ,d- \lfloor \frac{d}{2} \rfloor}^{n-d-3},S_1,S_1)) > Sz_{e}^*(G^{4}_{11}(\lfloor \frac{d}{2} \rfloor ,d- \lfloor \frac{d}{2} \rfloor,n-d-4)).$$
\end{lemma}

\begin{proof} Let $G=C_3(P_{\lfloor \frac{d}{2} \rfloor ,d- \lfloor \frac{d}{2} \rfloor}^{n-d-3},S_1,S_1)$ and $G'=G^{4}_{11}(\lfloor \frac{d}{2} \rfloor ,d- \lfloor \frac{d}{2} \rfloor,n-d-4)$. By the definition, one has
\begin{eqnarray*}Sz_{e}^{*}(G)-Sz_{e}^{*}(G')&=&\sum_{e=xy\in E(G) }m^{*}(e|G) -\sum_{e=xy\in E(G')}m^{*}(e|G')\nonumber\\
&=&\sum_{i=0}^{ \lfloor \frac{d}{2} \rfloor }(i+\frac{1}{2})(n-1-i+\frac{1}{2})+\sum_{i=0}^{ \lceil \frac{d}{2} \rceil }(i+\frac{1}{2})(n-1-i+\frac{1}{2})\\
&&+(n-2k-3)(0+\frac{1}{2})(n-1+\frac{1}{2})+2(1+\frac{1}{2})(n-2+\frac{1}{2})\\
&&+(1+\frac{n-2}{2})(1+\frac{n-2}{2})\\
&&-\sum_{i=0}^{ \lfloor \frac{d}{2} \rfloor }(i+\frac{1}{2})(n-1-i+\frac{1}{2})-\sum_{i=0}^{ \lceil \frac{d}{2} \rceil }(i+\frac{1}{2})(n-1-i+\frac{1}{2})\\
&&-(n-2k-4)(0+\frac{1}{2})(n-1+\frac{1}{2})-4(1+\frac{2}{2})(n-3+\frac{2}{2})\\
&=&\frac{1}{4}(n^{2}-18n+45) > 0.\end{eqnarray*}

Then, the proof is completed.
\end{proof}

\begin{lemma}\label{L74}
If $n>15$ and $d=4$ or $d=5$, then
$$Sz_{e}^*(G^{4}_{11}(\lfloor \frac{d}{2} \rfloor ,d- \lfloor \frac{d}{2} \rfloor,n-d-4)) > Sz_{e}^*(G^{4}_{32}(0,d-2,0,n-d-2, \lceil \frac{d}{2} \rceil)).$$
If $n>15$ and $d \geq 6$, then
$$Sz_{e}^*(G^{4}_{11}(\lfloor \frac{d}{2} \rfloor ,d- \lfloor \frac{d}{2} \rfloor,n-d-4)) < Sz_{e}^*(G^{4}_{32}(0,d-2,0,n-d-2, \lceil \frac{d}{2} \rceil)).$$

\end{lemma}

\begin{proof} Let $G=G^{4}_{32}(0,d-2,0,n-d-2, \lceil \frac{d}{2} \rceil)$ and $G'=G^{4}_{11}(\lfloor \frac{d}{2} \rfloor ,d- \lfloor \frac{d}{2} \rfloor,n-d-4)$.
By Lemma \ref{n1}, one has
\begin{eqnarray*}Sz_{e}^{*}(G)-Sz^{*}_{e}(G')&=&\sum_{e=xy\in E(G) }m(e|G) -\sum_{e=xy\in E(G')}m(e|G')\nonumber\\
&=&\sum_{i=1}^{ \lceil \frac{d}{2}  \rceil -1}i(n-1-i)+\sum_{j=4}^{ \lfloor \frac{d}{2} \rfloor  +1}j(n-1-j)+4(n-3)\\
&&-\sum_{i=1}^{ \lfloor \frac{d}{2}  \rfloor -1}i(n-1-i)-\sum_{i=1}^{ \lceil \frac{d}{2}  \rceil -1}i(n-1-i)-4(n-3)\\
&=&(2\lfloor \frac{d}{2} \rfloor -5)n-2\lfloor \frac{d}{2} \rfloor ^{2}-4\lfloor \frac{d}{2} \rfloor+18.\end{eqnarray*}

It can be checked that $( 2 \lfloor  \frac{d}{2} \rfloor -5)n-2\lfloor \frac{d}{2} \rfloor^{2}-4\lfloor \frac{d}{2} \rfloor+18 < 0$ for $\lfloor \frac{d}{2} \rfloor =2$
and $(2 \lfloor  \frac{d}{2} \rfloor -5)n-2\lfloor \frac{d}{2} \rfloor^{2}-4\lfloor \frac{d}{2} \rfloor+18  > 0$ for $\lfloor \frac{d}{2} \rfloor \geq 3$.
The result follows.
\end{proof}

By direct calculation, the following Lemmas \ref{L76}-\ref{L75} can be obtained.

\begin{lemma}\label{L76}
If $n>15$, then $Sz_{e}^*(G^{4}_{32}(0,2,0,n-6, 2)) < Sz_{e}^*(G^{4}_{32}(0,2,0,n-6, 1))$.
\end{lemma}

\begin{lemma}\label{L77}
If $n>15$, then
$$Sz_{e}^*(G^{4}_{32}(0,3,0,n-7, 3)) = \mbox{\rm{ min}}\{ Sz_{e}^*(G^{4}_{32}(0,3,0,n-7, i)), i=1, 2, 3  \}.$$
\end{lemma}

\begin{lemma}\label{L75}
If $n>15$ and $d=3$, then $Sz_{e}^*(C_3(P_{1,2}^{n-6},S_1,S_1) > Sz_{e}^*(C_4(S_{n-3},S_1,S_1,S_1))$.
\end{lemma}

\noindent{\bf Proof of Theorem \ref{Th1}:}

(i) By Theorems \ref{main-L-1} and \ref{main-L-2} and Lemmas \ref{L71} and \ref{L4-3-2}, (i) holds immediately.

(ii) If $d=n-3$, by Theorems \ref{main-L-1} and \ref{main-L-2} and Lemma \ref{L72}, the result holds.

(iii) If $6 \leq  d \leq n-4$, from Theorems \ref{main-L-1} and \ref{main-L-2} and Lemmas \ref{L73} and \ref{L74}, the result holds.

(iv) If $d=4$ or $d=5$, by Theorems \ref{main-L-1} and \ref{main-L-2} and Lemmas \ref{L73} and \ref{L74}, on has $G= G_{32}(0,d-2,0,n-d-2, i))$ for some $i \in [1, d-2]$. From Lemmas \ref{L4-3-2}, \ref{L76} and \ref{L77}, the result holds.

(v) If $d=3$, from Theorems \ref{main-L-1} and \ref{main-L-2} and Lemma \ref{L75}, the result is gotten directly.
\hspace{0.25cm}$\square$


\section{Conclusions}

In this paper, the graphs with minimum revised edge Szeged index among all the unicyclic graphs with given order and diameter are determined. For further study, it would be interesting to determine the extremal graph that has the minimum edge-vertex Szeged index of the unicyclic graphs with given order and diameter.

\section*{Acknowledgments}

This research is supported by National Natural Science Foundation of China (Nos.11971054, 11731002).



\begin{thebibliography}{99}





\vspace{-7pt} \bibitem{BONDY}
Bondy J, Murty U. Graph theory. In: Axler S, Ribet KA, editors. Graduate texts in
mathematics. New York: Springer; 2008.



\vspace{-7pt}\bibitem{Cai.X} X. Cai, B. Zhou, Edge Szeged index of unicyclic graphs, MATCH Commun. Math. Comput. Chem. 63 (2010) 133--144.


\vspace{-7pt}\bibitem{EWiener} P. Dankelmanna, I. Gutman, S. Mukwembi, H. C. Swart,
The edge-Wiener index of a graph, Discrete Math. 309 (2009) 3452--3457.




\vspace{-7pt}\bibitem{Do.RI} A. Dobrynin, R. Entringer, I. Gutman, Wiener index of trees: theory and applications, Acta Appl. Math. 66 (2001) 211--249.



\vspace{-7pt}\bibitem{Gut.A} I. Gutman, A formula for the Wiener number of trees and its extension to graphs containing cycles, Graph Theory Notes N. Y. 27 (1994) 9--15.


\vspace{-7pt}\bibitem{Gut.A.R} I. Gutman, A. R. Ashrafi, The edge version of the Szeged index, Croat. Chem. Acta. 81 (2) (2008) 263--266.







\vspace{-7pt}\bibitem{S.H.H.Y}
S. He, R.-X. Hao, A. Yu, On extremal cacti with respect to the edge Szeged index and edge-vertex Szeged index, Filomat. 32 (11) (2018) 4069--4078.


\vspace{-7pt}\bibitem{Kh.P} P. Khadikar, P. Kale, N. Deshpande, S. Karmarkar, V. Agrawal, Szeged indices
of hexagonal chains, MATCH Commun. Math. Comput. Chem. 43 (2000) 7--15.










\vspace{-7pt}\bibitem{HQLIU} H. Liu, X. Pan,
On the Wiener index of trees with fixed diameter, MATCH Commun. Math. Comput. Chem. 60 (2008) 85--94.


\vspace{-7pt}\bibitem{BKLY} F. Buckley, Mean distance in line graphs, Congr. Numer. 32 (1981) 153--162.




\vspace{-7pt}\bibitem{Yu} Y. Liu, A. Yu, M. Lu, R.-X. Hao,
On the Szeged index of unicyclic graphs with given diameter, Discrete Appl. Math. 233 (2017) 118--130.




\vspace{-7pt}\bibitem{EDGE.Relation}
M. J. Nadjafi-Arani, H. Khodashenas, A. R. Ashrafi, Relationship between edge
Szeged and edge Wiener indices of graphs, Glas. Mat. 47 (67) (2012) 21--29.




\vspace{-7pt}\bibitem{Shi} C. Ren, J. Shi,
On the Wiener index of unicyclic graphs with fixed diameter, J. East China Univ. Sci. Technol. 39 (2013) 768--772.

\vspace{-7pt}\bibitem{Tan}S. Tan,
The minimum Wiener index of unicyclic graphs with a fixed diameter, J. Appl. Math. Comput. (2016) 1--22.









\vspace{-7pt}\bibitem{zhang.H} H. Zhang, S. Li, L. Zhao, On the further relation between the (revised) Szeged index and the Wiener index of graphs, Discrete Appl. Math. 206 (2016) 152--164.


\vspace{-7pt}\bibitem{ZhouB.X} B. Zhou, X. Cai, Z. Du, On Szeged indices of unicyclic graphs, MATCH Commun. Math. Comput. Chem. 63 (2010) 113--132.


\vspace{-7pt}\bibitem{LMM} M. Liu, S. Wang, Cactus graphs with minimum edge revised Szeged index, Discrete Appl. Math. 247 (2018) 90--96.


\vspace{-7pt}\bibitem{LJP} J. Li, A relation between the edge Szeged index and  the ordinary Szeged index, MATCH Commun. Math. Comput. Chem. 70 (2013) 621--625.


\vspace{-7pt}\bibitem{HDZB} H. Dong, B. Zhou, C. Trinajsti\'c, A novel version of the edge-Szeged index, Croat. Chem. Acta. 84 (2011) 543--545.


\vspace{-7pt}\bibitem{YUERZ} A. Yu, K. Peng, R.-X. Hao, J. Fu, Y. Wang, On the revised Szeged index of unicyclic graphs with given diameter, Bull. Malays. Math. Sci. Soc.  43 (2020) 651--672.


\vspace{-7pt}\bibitem{AMCLSC} G. Wang, S. Li, D. Qi, H. Zhang, On the edge-Szeged index of unicyclic graphs with given diameter, Appl. Math. Comput. 336 (2018) 94--106.



\vspace{-7pt}\bibitem{HSJDAM} S. He, R.-X. Hao, A. Yu,  On the edge-Szeged index of unicyclic graphs with perfect matchings, Discrete Appl. Math. 284 (2020) 207--223.


\vspace{-7pt}\bibitem{WIENER}  H. Wiener, Structral determination of paraffin boiling points, J. Am. Chem. Soc. 69 (1947) 17--20.


\vspace{-7pt}\bibitem{LXLLMM}  X. Li, M. Liu, Bicyclic graphs with maximal revised Szeged index, Discrete Appl. Math. 161 (2013) 2527--2531.




\vspace{-7pt}\bibitem{RANDIC}  M. Randi\'c, On generalization of Wiener index for cyclic structures, Acta Chim. Slov. 49 (2002) 483--496.




\end{thebibliography}
\end{document}